\documentclass[a4paper]{amsart}

\usepackage{amsmath, array, amsfonts, amsthm, amssymb}
\usepackage{booktabs, array, xfrac, placeins, chngpage}

\newif\ifpdfAuthoring
\pdfAuthoringtrue

\ifpdfAuthoring
\usepackage[
bookmarks=true,
bookmarksnumbered=true,
bookmarksopen=false,
hypertexnames=false,
breaklinks=true,
colorlinks=true,
pagebackref=true,
hyperindex=true,
plainpages=false,
pdfauthor={Hatice Boylan and Nils-Peter Skoruppa},
pdftitle={A classical approach to relative quadratic extensions},
pdfsubject={Relative quadratic extensions},
pdfkeywords={Algebraic Number Theory, Relative Quadratic Extensions, Quadratic Congruences in Number Fields},
]{hyperref}
\fi

\newcommand    {\hilb}[3]       {\left(#2,#3\right)_{#1}}

\newcommand    {\C}             {{\mathbb C}}
\newcommand    {\Cl}            {\sym{Cl}}
\newcommand    {\units}[1]      {{#1}^\times}
\newcommand    {\D}             {{\mathfrak o}}
\newcommand    {\DL}            {\id O}

\newcommand    {\F}             {{\mathbb F}}

\newcommand    {\N}             {\sym{n}}
\newcommand    {\n}             {\sym{n}}

\newcommand    {\Q}             {{\mathbb Q}}

\newcommand    {\R}             {{\mathbb R}}
\newcommand    {\SL}[1][\D]     {\sym{SL}(2,{#1})}

\newcommand    {\Z}             {{\mathbb Z}}

\newcommand    {\dif}           {{\mathfrak{d}}}

\newcommand    {\dual}[1]          {{#1}^{\#}}

\newcommand    {\id}[1]         {{\mathfrak {#1}}}
\newcommand    {\isom}          {\cong}

\newcommand    {\leg}[2]        
                                {\genfrac{(}{)}{}{}{#1}{#2}}

\newcommand    {\mat}[4]        {\left(\begin{smallmatrix}#1&#2\\#3&#4\end{smallmatrix}\right)}

\newcommand    {\sym}[1]        {\operatorname{#1}}

\newcommand    {\tr}            {\sym{tr}}
\newcommand    {\uleg}[2]        
                                {\genfrac{[}{]}{}{}{#1}{#2}}

\newcommand    {\card}[1]       {\sym{card}\left(#1\right)}

\usepackage[most, breakable]{tcolorbox}
{\bigskip\begin{tcolorbox}[enhanced jigsaw, breakable ,size=title,
    colback=yellow!5!white,colframe=yellow!75!black,fonttitle=\bfseries,
    title={#1},pad at break=1mm, break at=\textheight/0pt]}
    {\end{tcolorbox}\bigskip}

\theoremstyle{plain}
\newtheorem{Theorem}{Theorem}
\newtheorem{Main Theorem}{Main Theorem}

\newtheorem{Proposition}[Theorem]{Proposition}
\newtheorem{Corollary}[Theorem]{Corollary}
\newtheorem{Lemma}[Theorem]{Lemma}
\theoremstyle{definition}

\theoremstyle{remark}
\newtheorem*{Remark}{Remark}

\begin{document}

\title[Relative quadratic extensions]{%
  A classical approach to relative quadratic extensions}

\author[Hat{\.I}ce Boylan]{%
  Hat{\.I}ce Boylan}

\address{{\.I}stanbul \"Universitesi, Fen Fak\"ultesi, Matematik
  B\"ol\"um\"u, 34134 Vezneciler, {\.I}stanbul, Turkey}

\email{hatice.boylan@gmail.com}

\author{Nils-Peter Skoruppa}

\address{Universit\"at Siegen, Department Mathematik, 57068 Siegen, Germany}

\email{nils.skoruppa@gmail.com}



\begin{abstract}
  We show that we can develop from scratch and using only classical
  language a theory of relative quadratic extensions of a given number
  field $K$ which is as explicit and easy as for the well-known case
  that~$K$ is the field of rational numbers. As an application we
  prove a reciprocity law which expresses the number of solutions of a
  given quadratic equation modulo an integral ideal $\id a$ of~$K$ in
  terms of $\id a$ modulo the discriminant of the equation. We study
  various $L$-functions associated to relative quadratic
  extensions. In particular, we define, for totally negative algebraic
  integers $\Delta$ of a totally real number field $K$ which are
  squares modulo~$4$, numbers $H(\Delta,K)$, which share important
  properties of classical Hurwitz class numbers.  In an appendix we
  give a quick elementary proof of certain deeper properties of the
  Hilbert symbol on higher unit groups of dyadic local number fields.
\end{abstract}

\maketitle

\tableofcontents

\newpage

\section{Introduction}
\label{sec:intro}

If $L$ is a quadratic extension of~$\Q$ we can pick any integral
quadratic irrationality~$a$ in $L$, set $\Delta:=\tr(a)-4\N(a)$ and
read off all important arithmetic properties of~$L$ from the rational integer
$\Delta$: We have $L=\Q(\sqrt \Delta)$, the integer $\Delta$ is a
square modulo~$4$, and if we divide out the largest perfect square
$f^2$ such that $\Delta_0:=\Delta/f^2$ is still a square modulo~$4$,
then $\Delta_0$ is the discriminant of $L$, which carries all
information about the ramification of rational primes in~$L$. The
application which associates to a rational prime $p$ the Legendre
symbol $\leg \Delta p$ extends to a character modulo~$\Delta$. It
factors through the primitive Dirichlet character $\leg
{\Delta_0}*$.
The Dedekind zeta functions~$\zeta_L(s)$ of $L$ and
$\zeta_\Q(s)=\zeta(s)$ of $\Q$ are related by the identity
$\zeta_L(s)=\zeta(s)L(\leg {\Delta_0}{*},s)$, where
$L(\leg {\Delta_0}{*},s)=\sum_{n\ge 1}\leg {\Delta_0}{n} n^{-s}$, and
this identity encodes the information of the splitting of the rational
primes in~$L$.

These facts are well-known since the beginning of algebraic number
theory in the 19th century, and the explicit character of this theory
is helpful in many applications. In contrast to this one does not find
such an easy and smooth description of the arithmetic theory of
relative quadratic extensions. One finds some hints towards such a
theory in Hecke's ``Vorlesungen \"uber die Theorie der Algebraischen
Zahlen''~\cite[\S 39]{Hecke} and, in particular, in its last
section~\cite[\S 63]{Hecke}. However, Hecke's treatment is not as
completely developed as in the case of extensions of~$\Q$. In modern
algebraic number theory relative quadratic extensions are subsumed
under the more abstract class field theory, which provides a
conceptual background for the arithmetic of general abelian
extensions, but lacks often explicitness even in simpler subclasses of
abelian extensions as for instance quadratic extensions.

In this article we propose an explicit theory of relative quadratic
extension which extends almost completely the classical and explicit
theory of quadratic extensions of~$\Q$. This depends, first of all, on
a correct extension of the notion of {\em discriminants} and {\em
  fundamental discriminants} to arbitrary number fields~$K$. By {\em
  discriminant in~$K$} we mean, similar to the case of rational
numbers, any algebraic integer in $K$ which is a square
modulo~$4$. However, if the class number of $K$ is larger than~$1$,
there will be in general no algebraic integer $f$ in $K$ whose square
divides $\Delta$ and such that $\Delta/f^2$ is still a square
modulo~$4$.  However, as it turns out, the maximal ideal
$\id f_\Delta$ whose square divides $\Delta$ and such that
\begin{equation*}
  \Delta \equiv x^2\bmod 4\id f_\Delta^2
\end{equation*}
for some algebraic integer $x$ in~$K$, serves very well as replacement
for the missing~$f$ for developing an explicit theory of relative
quadratic extensions.

This theory will be developed in \S\ref{sec:discriminants} and its
subsections. We show in particular, that, for a given discriminant
$\Delta$ in $K$, the relative discriminant $D_{L/K}$ of
$L=K(\sqrt \Delta)$ over $K$ is given by the formula
(Theorem~\ref{thm:discriminant-formula})
\begin{equation*}
  D_{L/K} = \Delta/\id f_\Delta^2
  .
\end{equation*}
We introduce a multiplicative function $\leg \Delta{\id a}$ on ideals
$\id a$ of $K$ which are relatively prime to $\Delta$ by setting, for
any prime ideal $\id p$,
\begin{equation*}
  \leg \Delta {\id p}
  =
  \begin{cases}
    +1&\text{if $\Delta$ is a square modulo~$4\id p$}
    \\
    -1&\text{otherwise.}
  \end{cases}
\end{equation*}
and show that it defines a Gr\"o\ss encharakter modulo~$\Delta$ whose
conductor equals $\Delta/\id f_\Delta^2$
(Theorem~\ref{thm:groessencharakter-property}). (For $\id p\nmid 2$ we
can replace $4\id p$ by $\id p$ since $\Delta$ is a square modulo~$4$,
and then the definition of the symbol $\leg\Delta{\id p}$ is
classical.) The associated primitive Gr\"o\ss encharakter is of course
nothing else than the Gr\"o\ss encharakter commonly known as `the
character of the quadratic extension $L/K$'; see \S2.6)

For a discriminant $\Delta$ in the field of rational numbers, the
Dirichlet $L$-function $L(\leg {\Delta}{*},s)$ and related Dirichlet
series like $\sum_{a\ge 1}N_\Delta(a)a^{-s}$, where $N_\Delta(a)$
counts the number of solutions $x$ modulo~$2a$ of
$x^2\equiv \Delta \bmod 4a$, play an important role in the arithmetic
in $\Q(\sqrt\Delta)$. In~\S\ref{sec:zeta-functions} we define and
study the analogues of these Dirichlet series for relative quadratic
extensions. The main results are Theorem~\ref{thm:counting-formula},
\ref{thm:zeta-delta-interpretation}, \ref{thm:analytic-properties} and
\ref{thm:special values}.

The first theorem states, for a given discriminant $\Delta$ and
integral ideal $\id a$ in a number field $K$ with ring of integers
$\D$, a reciprocity law expressing the number
\begin{equation*}
  N_{\Delta}(\id a)
  =
  \card {
    \left\{
      \D/2\id a
      :
      x^2\equiv \Delta\bmod 4\id a
    \right\}
  }
\end{equation*}
in terms of a generalized quadratic residue symbol $\chi_\Delta$
modulo $\Delta$ (see~\eqref{eq:chi-Delta} for its
definition). Theorem~\ref{thm:zeta-delta-interpretation} relates the
Dirichlet series
\begin{equation*}
  \zeta(\Delta,s) := \sum_{\id a} \frac {N_{\Delta}(\id a)}{\N(\id a)^{s}}
\end{equation*}
(the sum has to be taken over all integral ideals of $K$) to the
Dirichlet series of the order
$\DL_{\id f_\Delta} := \D +\D \frac {b+\sqrt \Delta}2$ of
$L=K(\sqrt \Delta)$, where $b$ is any solution in $\D$ of
$b^2\equiv \Delta \bmod 4\id f_\Delta^2$.
Theorem~\ref{thm:analytic-properties} summarizes the important
properties of the Dirichlet series
$L(\chi_\Delta,s)=\sum_{\id a} \chi_\Delta(\id a)\N(\id a)^{-s}$.
Finally, Theorem~\ref{thm:special values} proposes a generalization of
Hurwitz class numbers to totally negative discriminants in totally
really number fields $K$ as
\begin{equation*}
  H(\Delta,K):=L(\chi_\Delta,0)
  .
\end{equation*}
It gives an explicit finite formula for these numbers in terms of the
class numbers of $K$ and~$L=K(\sqrt \Delta)$, which shows in
particular, that these numbers are rational. These formulas generalize
the classical formulas for the Hurwitz class numbers.

In Appendix 1 (\S\ref{sec:tables}) the reader finds tables of the
numbers $H(\Delta,K)$ for various totally real number fields. We
computed, for all fields in the range of the Bordeaux tables of number
fields~\cite{nftables} with not more than $8$ genus classes, the
discriminants $\Delta$ (modulo squares in $K$) such that
$\Delta=\id f_\Delta^2$.  These discriminants are exactly those for
which the extension $K(\sqrt \Delta)/K$ is unramified at the finite
places, and their number modulo squares in $K$ equals the number of
genus classes in $K$. The results of our computation is summarized in
Table~\ref{tab:unit-dicriminants}.

Finally we prove in Appendix~2 (Section~\ref{sec:Hilbert-symbol}) an
orthogonality relation for the Hilbert symbol $\hilb K\_\_$ of a
dyadic number field~$K$ (cf.~Theorem~\ref{thm:dyadic-hilbert-symbol}),
which we needed in the study of the Gr\"o\ss encharakters
$\leg \Delta*$ in Theorem~\ref{thm:groessencharakter-property}. The
non-trivial part of its proof is to show, for the higher unit groups
$U_i$ in $K$, that $\hilb K {U_i}{U_j} =1$ whenever $i+j=2e$, where
$e$ is the ramification index of~$K$. All proofs of this identity
which we found in the literature make use of rather deep theorems of
local class field theory.  To comply with the explicit and
self-contained nature of this article we seeked for an elementary
proof. Eventually we were able to to give a simple six-line-proof
(Proof of Proposition~\ref{prop:Hilbert-symbol-on-unit-groups}), which
the reader might find amazing.

\section{Discriminants and relative quadratic extensions}
\label{sec:discriminants}

\subsection{Discriminants in a number field}
 
For this section we fix a number field $K$ with ring of integers
$\D$. Let $L$ be a quadratic extension of~$K$. We can obtain $L$ by
adjoining the square root $\Delta=\tr_{L/K}(a)^2-4\N_{L/K}(a)$ to $K$,
where $a$ is any integer in $L$ not in $K$. Note that $\Delta$ is a
square modulo~$4$ (not necessarily relatively prime to~$2$). For
obtaining the quadratic extensions of $K$ it suffices therefore to
adjoin square roots of integers $\Delta$ in $K$ which are squares
modulo~$4$.

We shall call those nonzero integers $\Delta$ of~$K$ which are squares
modulo~$4$ henceforth {\em discriminants in $K$}\footnote{The
  discriminants of $K$ which are relatively prime to~$2$ are
  occasionally called {\em Prim\"arzahlen} in older literature;
  see~\cite[\S59]{Hecke}.}.  For a discriminant $\Delta$ in $K$ we let
$\id f_\Delta$ be the largest integral ideal in $K$ whose square
divides $\Delta$ and such that
\begin{equation*}
  \Delta \equiv x^2 \bmod 4 \id f_\Delta^2
\end{equation*}
for some integer $x$ in $K$. We partition the discriminants of $K$
into classes, where $\Delta_1$ and $\Delta_2$ belong to the same class
if and only if $K(\sqrt {\Delta_1}) = K(\sqrt {\Delta_2})$.  It is
also clear that a class of discriminants coincides with the set of
discriminants in a class of $\units K/{\units K}^2$.  Moreover, we
have
\begin{Proposition}
  \label{prop:class-of-Delta}
  If $\Delta$ is not a square in $K$, then the class of $\Delta$
  coincides with the set of all nonzero $\tr_{L/K}(a)^2-4\N_{L/K}(a)$,
  where $a$ runs through the integers of $L=K(\sqrt {\Delta} )$.
\end{Proposition}
\begin{proof}
  If the discriminant $\Delta_1$ is in the class of $\Delta$, so
  that~$K(\sqrt {\Delta_1}) = K(\sqrt {\Delta})$ and
  $\Delta_1\equiv x^2\bmod 4$ for some integer $x$ of $K$, then
  $a:=\frac {-x+\sqrt {\Delta_1}}2$ is an integer of $L$, and
  $\Delta_1=\tr_{L/K}(a)^2-4\N_{L/K}(a)$.

  If vice versa $\Delta_1:=\tr_{L/K}(a)^2-4\N_{L/K}(a)\not=0$, then
  $L=K(a)$ (since, $a$ in $K$ implies $\Delta_1=0$), and hence $a$ is
  a solution of $x^2-\tr_{L/K}(a)x+\N_{L/K}(a)=0$, which implies
  $K(a)=K(\sqrt {\Delta_1})$.
\end{proof}

\subsection{The discriminant of a relative quadratic extension}

For an extension $L$ of $K$, we use $D_{L/K}$ for the relative
discriminant of the extension $L/K$.  For quadratic extensions,
$D_{L/K}$ equals the gcd of all numbers $\tr_{L/K}(a)^2-4\N_{L/K}(a)$,
where $a$ runs through the integers of~$L$. In other words, if we
write $L=K(\sqrt \Delta)$ with a discriminant of $K$, it equals the
ideal generated by all discriminants in the class of $\Delta$
(cf.~Proposition~\ref{prop:class-of-Delta}). The relative discriminant
of~$L/K$ can be expressed in terms of $\Delta$ alone by a formula
analogous to the case where~$K=\Q$.

\begin{Theorem}
  \label{thm:discriminant-formula}
  For any given integer $\Delta$ in $\units K$ (not necessarily a
  square mod~$4$), let~$\id s$ be the largest integral ideal whose
  square divides $\Delta$, and let $\id t$ be the largest integral
  ideal dividing $2$ such that $\Delta$ is a square mod
  $(\id s \id t)^2$. Then
  \begin{equation}
    \label{eq:discriminant-formula}
    D_{K(\sqrt \Delta)/K}
    =
    4\Delta/(\id s\id t)^2.
  \end{equation}
  In particular, if $\Delta$ is a discriminant in~$K$, then one has
  \begin{equation}
    \label{eq:proper-discriminant-formula}
    D_{K(\sqrt \Delta)/K} = \Delta/\id f_\Delta^2
    .
  \end{equation}
\end{Theorem}

\begin{proof}
  If $\Delta$ is a square in~$K$ both sides
  of~\eqref{eq:discriminant-formula} are equal to the unit ideal.  We
  can therefore assume that $\Delta$ is not a square, so that
  $L=K(\sqrt\Delta)$ is a quadratic extension of~$K$ with, say ring of
  integers~$\mathfrak{O}$. It suffices to prove
  $D_{L/K}\D_{\id p}=(\Delta/(\id s \id t)^2)\D_{\id p}$ for all prime
  ideals $\id p$ of~$K$, where $\D_{\id p}$ denotes the localization
  of~$\D$ at~$\id p$.

  Let $\id p$ be a prime ideal of~$K$. Then
  \begin{equation*}
    D_{L/K}\D_{\id p}=
    \det
    \begin{pmatrix}
      1&\omega\\
      1&\omega'
    \end{pmatrix}^2 \D_{\id p} = (\omega-\omega')^2\D_{\id p} ,
  \end{equation*}
  where $1,\omega$ is an $\D_{\id p}$-basis of
  $\mathfrak{O}\D_{\id p}$, and where $\omega'$ is the Galois
  conjugate of~$\omega$ in the extension~$L/K$ (see e.g.~\cite[\S3,
  Prop.~4]{Froehlich-65}).  For computing a basis we recall that the
  ring $\mathfrak{O}\D_{\id p}$ is the algebraic closure of
  $\D_{\id p}$ in $L$ (see e.g.~\cite[\S4, Lemma~1]{Froehlich-65}).
  In other words, we can write
  \begin{equation*}
    \mathfrak{O}\D_{\id p}
    =
    \left\{
      a+b\sqrt {\Delta_1}: a,b\in K,\ 
      2 a\in \D_{\id p},\ a^2-b^2\Delta_1 \in \D_{\id p} 
    \right\}
    ,
  \end{equation*}
  where we set $\Delta_1=\Delta/\pi^{2\lfloor l/2\rfloor}$ with
  $\id p^l$ denoting the exact power dividing $\Delta$, and $\pi$ an
  element of $\D_{\id p}$ such that $\D_{\id p}\id p=\pi \D_{\id p}$.
  It is quickly verified that we can take as second basis element
  $\omega$ of $\mathfrak{O}\D_{\id p}$ over $\D_{\id p}$ the element
  $\omega=\sqrt {\Delta_1}$ if ${\id p}$ is odd, and also if
  ${\id p}\mid 2$ and $\Delta_1$ is divisible by~$\id p$.  Otherwise
  one can take $\omega=(c+\sqrt \Delta_1)/\pi^{t}$, where $t$ is
  maximal such that $\pi^t\mid 2$ and $\Delta_1$ is a square
  modulo~$\pi^{2t}\D_{\id p}$, and where $c$ is any solution in
  $\D_{\id p}$ of $c^2\equiv \Delta_1\bmod \id p^{2t}$.  For this
  $\omega$ we have
  \begin{equation*}
    (\omega-\omega')^2
    =
    \begin{cases}
      4\Delta_1&\text{if $\id p$ is odd or $\id p\mid 2,\Delta_1$,}
      \\
      4\Delta_1/\pi^{2t}&\text{otherwise.}
    \end{cases}
  \end{equation*}
  The formula~\eqref{eq:discriminant-formula} becomes now obvious.

  Assume that $\Delta$ is a square mod~$4$. We have to show that
  $\id f:=\id s\id t/2$ is integral and that $\id f=\id f_{\Delta}$.
  For proving the first statement, it suffices to note that $\Delta$
  is a square mod~$4$ and mod~$(\id s\id t)^2$, and that $\id t$ is
  the largest ideal (dividing $2$) with this property.  For the second
  statement we note, first of all, that $\id f^2$ divides $\Delta$
  (since $(\id s\id t)^2$ divides $4\Delta$), and that $\Delta$ is a
  square mod $4\id f^2=(\id s\id t)^2$. But $\id f$ is also the
  maximal ideal with these properties. Namely, if, for some prime
  ideal $\id p$, the square of $\id f':=\id f\id p$ divides $\Delta$,
  then $v_{\id p}({\id t}\id p/2)\le 0$ (by the maximality of
  $\id s$), i.e.~$\id t\id p\mid 2$. But then $\Delta$ cannot be a
  square mod $4{\id f'}^2=(\id s \id t\id p)^2$ by the maximality of
  $\id t$. This completes the proof of the theorem.
\end{proof}

We note two corollaries.
\begin{Corollary}
  The discriminant of a relative quadratic extension represents a
  square in the class group.
\end{Corollary}
Indeed, the square of the class of $D_{K(\sqrt \Delta)/K}$ equals the
square of the ideal class of $\id f_\Delta$.  The corollary is due to
Hecke~\cite[\S63, Satz~177]{Hecke}, who proved it in fact for
arbitrary (not necessarily quadratic) relative extensions.

For the $\id f_\Delta$ we find the following.
\begin{Corollary}
  \label{prop:fD-class-invariance}
  The $\id f_\Delta$, for $\Delta$ in a given class
  of~$\units K/{\units K}^2$ belong all to one and the same ideal
  class of $K$. More precisly, one has
  \begin{equation*}
    \id f_{a^2\Delta}=a\id f_\Delta
  \end{equation*}
  for any $\Delta$ and $a$ in $K$ such that $\Delta$ and $a^2\Delta$
  are discriminants.
\end{Corollary}

\begin{proof}
  If $\Delta$ and $\Delta'$ are in the same class
  of~$\units K/{\units K}^2$, then,
  using~\eqref{eq:proper-discriminant-formula}, we find
  $\id f_\Delta^2=a^2\id f_{\Delta'}$ for some $a$ in~$\units K$,
  which implies $\id f_\Delta=a\id f_{\Delta'}$.  However, one can
  prove the corollary also more directly as follows.  If $a$ is a
  nonzero integer in $K$ then $\id f_{a^2\Delta}=a\id f_\Delta$ (since
  $a\Delta$ is divisible by the square of $a\id f_\Delta$ a square
  modulo $4a^2\id f_\Delta^2$, and on the other hand
  $\id f_{a^2\Delta}$ can obviously not larger than $a\id
  f_\Delta$).
  Therefore, if $\Delta_i$ ($i=1,2$) are in the same class, i.e.~if
  $a_1^2\Delta_1=a_2^2\Delta_2$ for integers $a_i$, we conclude that
  $a_1\id f_{\Delta_1}=a_2\id f_{\Delta_2}$.
\end{proof}

In fact, one can say more about the $\id f_{\Delta}$. The ring of
integers $\id O$ of $K(\sqrt \Delta)$, being in general not free as
module over $\D$, can be written as $\id O=\id g a \oplus \D b$ with a
suitable fractional ideal $\id g$ and $a$, $b$ in $\id O$ if $\Delta$
is not a square( and otherwise as $\id O = \D = \id g a$).  The ideal
$\id g$ is not unique, but its ideal class is, which is called the
{\em Steinitz invariant of the $\D$-module
  $\id O$}~\cite[Thm. 13]{Froehlich-Taylor}.

\begin{Proposition}
  \label{prop:Steinitz-class}
  The $\id f_\Delta^{-1}$, for $\Delta$ in a given class
  of~$\units K/{\units K}^2$ belong to the Steinitz invariant of the ring
  of integers of $K(\sqrt \Delta)$ as module over $\D$.
\end{Proposition}

\begin{proof}
  The case that $\Delta$ is a square being trivial we assume that
  $L=K(\sqrt \Delta)$ has degree two over~$K$.  For a prime ideal
  $\id p$ of $K$ we have
  $\id O\D_{\id p} = \id g\D_{\id p}a+\D_{\id p}b$, and as in the
  proof of Theorem~\ref{thm:discriminant-formula} we find therefore
  $D_{L/K}\D{\id p}=\pi^{2g}(ab'-a'b)^2$, where
  $\pi\D_{\id p}=\id p\D_{\id p}$, and
  $\id g\D_{\id p} = \pi^{2g}\D_{\id p}$. It follows
  $D_{L/K}=\id g^2 (ab'-a'b)^2$. On the other hand
  $D_{L/K}=\Delta/\id f_{\Delta}$, and hence
  $(\id g/\id f_\Delta)^2=\Delta/(ab'-a'b)^2$. But
  $\Delta/(ab'-a'b)^2$ is a square in $\units K$ (since $ab'-a'b$,
  being not invariant under the Galois group of $L/K$, is in $L$ but
  not in $K$), and hence $\id g/\id f_\Delta$ is a principal
  ideal. This proves the proposition.
\end{proof}

\subsection{Fundamental discriminants}

There is still a dichotomy left since the notion ``Fundamental
discriminant'' is missing for general relative quadratic
extensions. More precisely, we would like to have, for a given class
$C$ in $\units K/{\units K}^2$, a discriminant~ $\Delta_ 0 $ such that
$\Delta_0f^2$ runs through all discriminants in the given class when
$f$ runs through all nonzero integers of $K$. Such a $\Delta_0$ would
be uniquely determined by this property modulo ${\units \D}^2$, and it
would generate the discriminant ideal $D_{L/K}$ of the quadratic
extension of~$K$ determined by~$C$.

Such a $\Delta_0$ will not in general exist. In fact, it is not hard
to see that it exists if and only if the ideal class of the
$\id f_\Delta$ ($\Delta$ in~$C$) is trivial (see
Theorem~\ref{thm:criteria-for-principal}). However, such a $\Delta_0$
exists for any class~$C$ as an id\`ele of~$K$, and it is useful to
study its id\`elic construction since it will give us, amongst
others, further criteria for the existence of fundamental discriminant
for classes in $\units K/{\units K}^2$.

For a valuation~$v$ of~$K$ let $K_v$ denote the completion of~$K$ at
$v$, and $\D_v$ the ring of integers in~$K_v$ (with the convention
$\D_v=K_v$ if $v$ is real or complex). We use~$I$ for the id\`ele
group of~$K$, and $U$ for the direct product over all $\units \D_v$.
Finally we identify~$K$ with its image in $I$ under the diagonal
embedding.

Let $C$ be a class in~$\units K/{\units K}^2$. Let $D_CU^2$ be the
element of
\begin{equation*}
  H_K:=\units K I^2/U^2
\end{equation*}
which at a finite place~$v$ is defined as
\begin{equation}
  \label{eq:fundamental-discriminant}
  (D_C)_v = \Delta/\pi_v^{2v(\id f_{\Delta})}
  ,
\end{equation}
at a real place as $(D_C)_v = \sym{sign}\left(\sigma(\Delta)\right)$
with the embedding $\sigma$ corresponding to~$v$, and at a complex
place as~$1$. Here $\pi_v$ is a prime element of~$K_v$ (and therefore
unique up to multiplication by a unit in~$\units \D_v$), and $\Delta$
is an element of~$C$. Note that the coset $(D_C)_v{\units \D_v}^2$
does not depend on the choice of~$\pi$.

\begin{Lemma}
  The coset $D_CU^2$ does not depend on the choice of~$\Delta$ in~$C$.
\end{Lemma}

\begin{proof}
  We have $D_C=\Delta/\varphi^2$, where $\varphi$ is the id\`ele
  associated to $\id f_{\Delta}$,
  i.e.~$\varphi_v=\pi_v^{v(\id f_{\Delta})}$. If $\Delta'$ is another
  discriminant in~$C$, then~\eqref{eq:proper-discriminant-formula}
  implies $\Delta/\varphi^2=\varepsilon\Delta'/\varphi'^{2}$, where
  $\varphi'$ is the id\`ele associated to~$\id f_{\Delta'}$ and
  $\varepsilon$ is in~$U$. Since $\Delta$ and $\Delta'$ differ by a
  square we conclude that $\varepsilon$ is in fact in $U^2$, which
  proves the lemma.
\end{proof}

We call any representative $D_C$ of $D_CU^2$ (by slight abuse of
language) {\em the fundamental discriminant of~$C$}. This is justified
by the following propositions.

First of all, as for the case that $K$ is the field of rational
numbers the fundamental discriminant $D_C$ uniquely determines $C$
(and hence the extension~$K(\sqrt \Delta)$) since
$C = D_CI^2\cap \units K$, in other words:

\begin{Proposition}
  The map $C\rightarrow D_CU^2$ defines an injection
  $\units K/{\units K}^2\rightarrow H_K$.
\end{Proposition}

Moreover, the usual properties of rational fundamental discriminants
have also their equivalents as follows:

\begin{Proposition}
  \label{eq:fundamental-discriminant-properties}
  Let $C$ be a class in~$\units K/{\units K}^2$ and $D_C$ its
  fundamental discriminant.
  \begin{enumerate}
  \item Every discriminant $\Delta$ in~$C$ can be written uniquely
    mod~$U^2$ as $\Delta=D_C\alpha^2$ with an id\`ele $\alpha$ whose
    valuation at a finite place is non-negative.
  \item $D_C$ is mapped onto~$D_{K(\sqrt \Delta)/K}$ under the natural
    map which takes $I/U^2$ onto the group of fractional
    ideals\footnote{This is the map which takes an $\alpha U^2$ to the
      product of all~$\id p^{v(\alpha_v)}$, where $\id p$ runs through
      the prime ideals of~$K$ and $v$ is the place corresponding
      to~$\id p$.} of~$K$.
  \item $(D_C)_v$ is a square mod~$4\D_v$ at every finite place
    of~$K$,
  \item For every real place the sign of $(D_C)_v$ equals the sign of
    $\sigma(\Delta)$ for every discriminant $\Delta$ in~$C$, with the
    embedding $\sigma$ of $K$ corresponding to~$v$.
  \end{enumerate}
\end{Proposition}
\begin{proof}
  These statements are immediate consequences of the
  definition~\eqref{eq:fundamental-discriminant} of~$D_C$, where, for
  the second one, one needs also the
  formula~\eqref{eq:proper-discriminant-formula}.
\end{proof}

It is not difficult to verify that our~$D_C$ coincides with the
enhanced notion of relative discriminant for arbitrary extensions as
proposed in~\cite{Froehlich-60} (whose definition is slightly
different from the one given here).
 
We call $D_C$ {\em principal} if it is represented by a number
in~$\units K$, i.e.~if $D_CU^2=\Delta_0 U^2$ for some number
$\Delta_0$. This number is then a discriminant in~$C$ (by~(3) of the
preceding proposition). Moreover, if $D_C$ is principal, say
represented by $\Delta_0$, then~$\Delta_0 a^2$ runs through all
discriminants in $C$ when $a$ runs through the nonzero integers
of~$K$. (Namely, if $\Delta$ is a discriminant in~$C$ then
$\Delta/\Delta_0=\alpha^2$ for some~$\alpha$ in~$I$ as in~(1), and on
the other hand~$\Delta/\Delta_0$ is in~${\units K}^2$, so that
$\alpha$ is in fact in $\D$.)

\begin{Theorem}
  \label{thm:criteria-for-principal}
  Let $C$ be a class in $\units K/{\units K}^2$, let $D_C$ denote its
  fundamental discriminant, and let $\Delta$ in~$C$. Then the
  following statements are equivalent:
  \begin{enumerate}
  \item $D_C$ is principal.
  \item The ideal $\id f_\Delta$ is principal.
  \item The ring of integers of $K(\sqrt \Delta)$ is a free module
    over~$\D$.
  \end{enumerate}
\end{Theorem}
\begin{proof}
  (1) implies $\Delta=D_Ca^2$ for some integer~$a$ in $K$. Since
  $D_C\D=D_{K(\sqrt \Delta)/K}$ (by
  Prop.~\ref{eq:fundamental-discriminant-properties}~(2)) and
  $\Delta=D_{K(\sqrt \Delta)/K}\id f_{\Delta}$ we conclude
  $\id f_{\Delta}=a\D$. Vice versa, if $\id f_\Delta=a\D$ for some
  integer~$a$, then $\Delta/a^2{\units \D_v}^2=D_CU^2$ for every $v$.
  (2) and (3) are equivalent since by
  Proposition~\ref{prop:Steinitz-class} the ring of integers of
  $K(\sqrt \Delta)$ is isomorphic as $\D$-module
  to~$\id f_\Delta^{-1}\oplus \D$.
\end{proof}
 


For the reader who wishes to avoid the id\`elic setting for the notion
of fundamental discriminants the following remark may be useful. The
natural map
$H_K\rightarrow \units K I^2/I^2 \isom \units K/{\units K}^2$ defines
an exact sequence
\begin{equation*}
  1\rightarrow
  I^2/U^2
  \rightarrow
  H_K
  \rightarrow
  \units K/{\units K}^2
  \rightarrow 1
  .
\end{equation*}
This sequence splits, a section is given by $C \mapsto D_CU^2$. We
therefore obtain a isomorphism
$H_K\isom I^2/U^2 \times \units K/{\units K}^2 \isom J_K^2 \times
\units K/{\units K}^2=:H_K'$,
where $J_K$ is the group of nonzero fractional ideals of~$K$. (For the
latter isomorphism note that the natural map
$I^2/U^2\rightarrow \left(I/U\right)^2$ is an isomorphism since
$I^2\cap U=U^2$.) We leave it to the interested reader to go again
through this section while replacing $H_K$ systematically by~$H_K'$.

\subsection{The Gr\"o{\ss}encharakter of a relative quadratic extension}

We extend the Dirichlet characters $\leg\Delta*$ from the theory where
$K$ is the field of rational numbers to arbitrary number fields $K$ as
follows: Let $\Delta$ be a discriminant of~$K$. For a prime ideal
$\id p \nmid \Delta$, we set
\begin{equation*}
\leg \Delta {\id p}
=
\begin{cases}
+1&\text{if $\Delta$ is a square modulo~$4\id p$}
\\
-1&\text{otherwise.}  
\end{cases}
\end{equation*}
Of course, for $\id p \nmid 2$ the number $\Delta$ is a square
modulo~$4\id p$ if and only if it is square modulo~$\id p$ as follows
from the Chinese remainder theorem.  We continue $\leg {\Delta}{*}$ to
a homomorphism of the group of fractional ideals relatively
prime\footnote{A fractional ideal is called {\em relatively prime} to
  $\Delta$ is it is of the form $\id a/\id b$ with integral ideals
  $\id a$ and $\id b$ both of which have no prime ideal common
  with~$\Delta$.} to $\Delta$ onto the group~$\{\pm1\}$.

\begin{Theorem}
  \label{thm:groessencharakter-property}
  The homomorphism $\leg {\Delta} {*}$ defines a Gr\"o{\ss}encharakter
  modulo~$\Delta$ of infinity type
  $\alpha\mapsto \prod_{\sigma\in M}\sym{sign}\sigma(\alpha)$, where
  $M$ is the set of real embeddings of $K$ with $\sigma(\Delta)<0$.
  Its conductor equals $\Delta/\id f_\Delta^2$.
\end{Theorem}

We postpone the proof to the end of this section.

According to the theorem the character $\leg \Delta *$ is the
restriction of a primitive Gr\"o{\ss}encharakter
modulo~$\Delta/\id f_\Delta^2$, which we denote in the sequel by
$\uleg \Delta *$. Note that $\uleg \Delta *$ is uniquely determined
by~$\leg \Delta *$.  Indeed, if the fractional ideal $\id a$ is
relatively prime to $\id D:=\Delta/\id f_\Delta^2$, then we can find
an $a$ in~$\units K$ relatively prime to $\id D$ and such that
$\id b:=\id a/a$ is relatively prime to~$\Delta$. (One can take, for
instance, for $\id b$ any prime ideal in the same ideal class as
$\id a$ which doe not divide $\Delta$). But then
$\uleg \Delta {\id a}=\leg \Delta {\id b}\uleg \Delta {a\D}$, and
$\uleg \Delta {a\D}\prod_{\sigma\in M}\sym{sign}\sigma(a)$ depends
only on $a$ modulo~$\id D\D_{\id D}$ (where $\D_{\id D}$ is the
intersection of the localizations~$\D_{\id p}$ of $\D$ at $\id p$
($\id p \mid \id D$)), and hence can be evaluated by replacing~$a$ by
any~$b$ in $a+\id D\D_{\id D}$ which is relatively prime to~$\Delta$.

\begin{Proposition}
  The primitive Gr\"o{\ss}encharakter $\uleg \Delta *$ depends only
  on the class of~$\Delta$ in~$\units K/{\units K}^2$.
\end{Proposition}

\begin{proof}
  Let $\Delta'$ be another discriminant in the same class as
  $\Delta$. Then $\Delta a^2=\Delta' a'^2$ with suitable nonzero
  integers $a$ and $a'$. Let $\Delta''= \Delta a'^2=\Delta' a^2$. It
  suffices to show that $\uleg {\Delta''}* = \uleg \Delta *$ and
  $\uleg {\Delta''}* = \uleg {\Delta'} *$. But this is clear since
  $\leg {\Delta''}*$ is the restriction of~$\leg \Delta *$ and of
  $\leg {\Delta'} *$ to the group of fractional ideals relatively prime
  to $\Delta''$.
\end{proof}

\subsection{The decomposition of primes in~relative quadratic extensions}

Let
\begin{equation*}
  L\left(\uleg \Delta*, s\right)
  =
  \sum_{\id a}
  \uleg \Delta{\id a}\,\N(\id a)^{-s}
  ,
\end{equation*}
where the sum runs over all integral ideals of $K$ with the convention
that $\uleg \Delta{\id a}=0$ if $\id a$ is not relatively prime to
conductor of $\uleg \Delta*$.
\begin{Theorem}
  \label{thm:decomposition-law}
  If $\Delta$ is not a square in $K$, then
  \begin{equation*}
    \zeta_{K(\sqrt \Delta)}(s)=\zeta_K(s)L\left(\uleg \Delta*, s\right)
  \end{equation*}
\end{Theorem}
\begin{proof}
  It is a basic fact~\cite[Satz 117]{Hecke} that every prime ideal
  $\id p$ of $K$ is inert (i.e.~remains a prime in $L$), or splits
  (i.e.~factors into a product of two different primes), or ramifies
  (i.e.~is the square of a prime ideal in $L$). Moreover, which
  property holds true is given by the {\em character criterion},
  namely, the first, second or third property holds true accordingly
  as $\uleg \Delta{\id p}$ equals $-1$, $+1$ or~$0$. For
  $\id p\nmid\Delta$ the character criterion is~\cite[Satz 118,
  119]{Hecke} (for applying Satz 119 loc.cit.~recall that
  $\leg \Delta{\id p}$ equals $+1$ or $-1$ according as $\Delta$ is a
  square mod $4\id p$, and that $\Delta$ is a square mod~$4$). But
  then the criterion is also true for any
  $\id p\nmid D_{L/K} = \Delta/\id f_\Delta^2$: If~$\id p \mid \Delta$,
  but~$\id p\nmid \Delta/\id f_\Delta^2$, we can find a discriminant
  $\Delta'$ in $\Delta {\units K}^2$ with $\id p\nmid \Delta'$ and
  apply Satz~118,~119 loc.cit.~to $\Delta'$. Indeed, choose an
  integral ideal $\id b$ in the class of~$\id f_\Delta$ relatively
  prime to $2\id f_{\Delta}$, let $a=\id b\id f_{\Delta}$, and set
  $\Delta'=\Delta a^2$.

  If $\id p\mid \Delta/\id f_{\Delta}^2$, i.e.~if
  $\uleg \Delta{\id p}=0$, then $\id p$ is ramified according to Satz
  118 loc.cit.~if the exact $\id p$-power dividing
  $\Delta/\id f_{\Delta}^2$ is $\id p^f$ with odd $f$ (which is always
  the case for $\id p \nmid 2$). If $f$ is even and $\id p \mid 2$ the
  ideal $\id p$ ramifies in~$L$ by Satz 119 loc.cit.~(for applying
  Satz 119 we write $L=K(\sqrt d)$ with some integer $d$
  in~$\Delta{\units K}^2$ which is not divisible by~$\id p$ and
  observe that $d$ cannot be a square modulo~$4$ since otherwise
  $\Delta/\id f_\Delta^2 = d/\id f_d^2$, contradicting $\id p\nmid d$;
  one can choose $d=\Delta a^2$ with $a=\id b/\id f_\Delta\id p^{f/2}$
  and $\id b$ relatively prime to~$\id p$).

  The claimed identity is now an easy consequence of the character
  criterion by comparing, for each prime ideal ${\id p}$ of~$K$, the
  $\id p$th Euler factors on both sides of the claimed identity.
\end{proof}

If the conductor of $\uleg \Delta*$ is $1$ no prime ideal ramifies
and vice versa. In other words, we have
\begin{Corollary}
  The extension $K(\sqrt \Delta)$ is unramified if and only if
  $\Delta=\id f_\Delta^2$.
\end{Corollary}

\subsection{Proof of Theorem~\ref{thm:groessencharakter-property}}
\label{sec:proof-groessencharakter-property}

It remains to prove Theorem~\ref{thm:groessencharakter-property}.  The
educated reader might have noticed that $\uleg \Delta*$ is
essentially nothing else than the Artin reciprocity map associated to
$K(\sqrt \Delta)$.

Indeed, let $\Delta$ be a discriminant not a square, let
$L=K(\sqrt \Delta)$, let~$\mathfrak{\id O}$ be the ring of integers
of~$L$ and let $\sigma$ be the nontrivial Galois substitution
of~$L/K$, which maps $\sqrt \Delta$ to $-\sqrt\Delta$. As we saw in
the proof of Theorem~\ref{thm:decomposition-law}, any prime ideal
${\id p}$ of $K$ relatively prime to $D_{K(\sqrt \Delta)/K}$ with
$\leg \Delta{\id p}=1$ factors in $L$ in the form
${\id p}\mathfrak{\id O}=\id P\sigma(\id P)$ with
$\id P\not=\sigma(\id P)$. The {\em Frobenius $F_{\id p}$} (i.e.~the
generator of the subgroup of $\sym{Gal}(L/K)$ mapping each prime ideal
of $L$ over $\id p$ to itself) is therefore the identity. On the other
hand, if $\leg \Delta{\id p}=-1$, then $\id p\mathfrak {O}$ is the
only prime ideal in $L$ over $\id p$, and hence $F_{\id p}=\sigma$.
Therefore, if $f$ denotes the isomorphism of $\{\pm 1\}$ with
$\sym{Gal}(L/K)$, then $f\circ\leg\Delta*$ equals the Artin
reciprocity map $\leg {L/K}{\id p}$ on the group of fractional ideals
relatively prime to~$\Delta$, which maps a prime ideal ${\id p}$ to
$F_{\id p}$.  The fact that $\leg \Delta*$ is a Gr\"o{\ss}encharakter
follows then from well-known facts of Global Classfield theory, which
include also that the discriminant $D_{L/K}$ is the conductor of
$\leg \Delta*$~\cite[Ch. VI \S 4.4]{Cassels-Froehlich}.

However, we prefer to keep with the explicit character of this note
and to give a more direct and elementary proof. We shall need the
product formula for the quadratic Hilbert symbols though, which
however can be proved without developing full Classfield Theory (see
e.g.~\cite[Ch.~7]{OMeara}).

\begin{proof}[Proof of Theorem~\ref{thm:groessencharakter-property}]
  For a nonzero $a$ in $\units K$ relatively prime to~$\Delta$, set
  \begin{equation*}
    \psi(a)
    := \leg {\Delta} {a\D} \prod_{\sigma\in M}\sym{sign} \sigma(a)
    .
  \end{equation*}
  Note that $\psi$ defines a linear character of the multiplicative
  group $\units K_\Delta$ of elements in~$\units K$ relatively prime
  to~$\Delta$.  That $\leg\Delta*$ is a Gr\"o{\ss}encharakter
  mod~$\Delta$ means that $\psi$ factors through a homomorphism of
  $\units{\left(\D/\Delta\D\right)}$, i.e.~that $\psi$ is trivial on
  the kernel of the natural map
  $\units K_\Delta\rightarrow \units {\left(\D/\Delta\D\right)}$
  (which sends $a$ to $a'+\Delta\D$, where $a'$ is any element in $\D$
  such that $a'\equiv a \bmod \Delta K_\Delta$, and where $K_\Delta$
  is the ring of elements in $K$ relatively prime to $\Delta$).  For
  this it suffices to show that, for integral $a$, the value $\psi(a)$
  depends only on the residue class of $a$ modulo~$\Delta$ (as one
  sees on writing any $a$ in the kernel of the natural map in the form
  $a=\gamma/\delta$ with integers $\gamma\equiv \delta\bmod \Delta$
  and $(\delta,\Delta)=1$.)
  
  So, let $a$ be integral and relatively prime to~$\Delta$.  For the
  proof that $\psi(a)$ depends only on $a$ modulo~$\Delta$, we use the
  product formula (see e.g.~\cite[Ch.~VI, Thm.~(8.1)]{Neukirch}
  or~\cite[71:18 Thm.]{OMeara})
  \begin{equation}
    \label{eq:product-formula}
    \prod_v \hilb v \Delta a = 1 ,
  \end{equation}
  where $v$ runs through all places of $K$ (including the infinite
  ones) and $\hilb v {\_} {\_}$ denote the quadratic Hilbert symbol of
  the completion~$K_v$ of~$K$. Thus, for $a,b$ in $K_v$, we have
  $\hilb v ab=+1$ if $ax^2+by^2=1$ has a solution $x$, $y$ in $K_v$
  and $\hilb v ab=-1$, otherwise.

  If $v$ is infinite, corresponding to the embedding $\sigma$ of~$K$
  into~$\C$, then obviously $\hilb v \Delta a = -1$ if and only if
  $\sigma$ is real and $\sigma(\Delta)$ and $\sigma(a)$ are both
  negative. In other words, the contribution to the left hand side
  of~\eqref{eq:product-formula} of the infinite places
  equals~$\prod_{\sigma\in M}\sym{sign} \sigma(a)$. If $v$ is a finite
  place corresponding to a prime ideal~$\id p$ not dividing~$2$, then
  $\hilb v \Delta a = 1$, or $\leg \Delta {\id p}^{v(a)}$, or
  $\leg a {\id p}^{v(\Delta)}$ according as $\id p \nmid \Delta$,
  or~$\id p \mid a$, or~$\id p\mid \Delta$ (as follows
  e.g.~from~\cite[Ch.~V, Prop.~(3.4)]{Neukirch} or~\cite[p.~165,
  63:11a]{OMeara}).  Therefore the product
  formula~\eqref{eq:product-formula} becomes
  \begin{equation*}
    \leg \Delta {\id a}\leg {a}{\id D}
    \prod_{\sigma\in M}\sym{sign} \sigma(a)
    \prod_{v\mid 2} \hilb v \Delta {a}
    = 1,
  \end{equation*}
  where $\id D$ is the odd part of~$\Delta$ (i.e.~the product of all
  prime ideal powers dividing~$\Delta$ and relatively prime to~$2$),
  and $\id a$ is the odd part of~$a$. In other words
  \begin{equation*}
    \psi(a) = \leg \Delta {a/\id a} \leg a {\id D}
    \prod_{v\mid 2} \hilb v \Delta {a}
    =
    \leg a {\id D}
    \prod_{v\mid 2} \leg \Delta {\id p_v^{v(a)}} \hilb v \Delta {a}
    ,
  \end{equation*}
  where $\id p_v$ is the prime ideal corresponding to~$v$. If
  $v(\Delta)=0$ the $v$th factor on the right equals~$1$
  (see~\cite[Ch.XV,\S3,Prop.~6]{Serre}).

  Finally, let $v\mid 2$ and $l:=v(\Delta)\ge 1$ (and hence
  $v(a)=0$). We need to prove that the quadratic character of
  $U:=\units \D_v$ defined by
  \begin{equation*}
    \kappa:a\mapsto \hilb v \Delta {a}
  \end{equation*}
  factors through a Dirichlet character modulo~$\id p_v^{l}$. Since
  the natural map
  $\units \D_v\rightarrow \units {\left(\D/\id p_v^l\right)}$ is
  surjective and has kernel $U_l:=1+\id p_v^l\D_v$ the character
  $\kappa$ factors through a Dirichlet character mod~$\id p_v^l$ if
  and only if $\kappa$ is trivial on $U_l$.

  Since $U_{2e+1}\subseteq U^2$, where $e=v(2)$, i.e.~by the Local
  Square Theorem (see Section~\ref{sec:Hilbert-symbol}), $\kappa$ is
  in any case a Dirichlet character mod~$\id p_v^{2e+1}$.  Hence we
  can assume that $l\le 2e$. But then $l$ is even (since $\Delta$ is a
  square mod~$4$). Let $\pi_v$ a uniformizer of $K_v$ and write
  $\Delta=\pi_v^lb$. By assumption $b$ is a square
  modulo~$\id p_v^{2e-l}$, i.e.~$b/c^2 \equiv 1\bmod \id p_v^{2e-l}$.
  In other words, $\Delta = \pi_v^lc^2d$, where $d$ is in~$U_{2e-l}$,
  with the convention $U_0=U$ for the case $l=2e$. It follows that
  $\kappa(a) = \hilb v d a$.  It remains therefore to show that, for
  any even integers $i,j\ge 0$ with $i+j=2e$, one has
  $\hilb v {U_i}{U_j} =1$. Though this seems to be known (see the very
  last exercise~\cite[Ch.XV, \S3, Ex.~3]{Serre}, from which it
  follows) we did not find any reference for this providing a proof
  not depending on heavy machinery (see~\cite{Dalawat-reciprocity},
  where a proof of a more general result is given based on Local Class
  Field Theory). For the convenience of the reader we give a
  self-contained and easy proof in the
  Appendix Section~\ref{sec:Hilbert-symbol},
  Proposition~\ref{prop:Hilbert-symbol-on-unit-groups}.
   
  The conductor of $\leg \Delta *$ equals the conductor of the
  Dirichlet character~$\psi$ (see e.g.~\cite[Ch.~7, (6.2)
  Prop.]{Neukirch}). But this is the product all local conductors
  $\id p_v^{s_v}$ of the Dirichlet characters
  $a+\id p_v\mapsto \hilb v \Delta a$, taken over all finite $v$ with
  $v(\Delta)\ge 1$. If $\id p_v\nmid 2$ then obviously $s_v=0$ or
  $s_v=1$ according as $v(\Delta)$ is even or odd, and therefore
  $s_v=v(\Delta/\id f_\Delta^2)$. If on the contrary $\id p_v \mid 2$
  then $\hilb v\Delta a = \hilb v {\Delta/\pi_v^{2k}}a$ for any
  integer~$k$, and as we saw $a\mapsto \hilb v {\Delta/\pi_v^{2k}}a$
  factors through a Dirichlet character modulo $\Delta/\pi_v^{2k}$ if
  $\Delta/\pi_v^{2k}$ is integral and a square modulo~$4$. Therefore,
  $s_v$ is $\le $ the largest power $\id p_v^{2k}$ dividing $\Delta$
  and such that $\Delta':=\Delta/\id p_v^{2k}$ is a square modulo~$4$,
  i.e.~$s_v\le v(\Delta')$.  But $s_v$ is even equal to
  $v(\Delta')$. For this let $l=v(\Delta')$. Clearly, $l\le 2e+1$
  (with $e=v(2)$). If $l=2e+1$,
  Theorem~\ref{thm:dyadic-hilbert-symbol} in
  Section~\ref{sec:Hilbert-symbol}below implies
  $\hilb v\Delta a=\hilb v \pi a$ for all units $a$, and another
  application of this theorem implies that there is a unit $a$ in
  $U_{4e}$ such that $\hilb v\pi a=-1$; we conclude $s_v=2e+1$.  If
  $l\le 2e$ (so that $l$ is even), we proceed as in the last paragraph
  and write as before $\Delta'=\pi_v^lc^2d$ with units $c$, $d$ and
  $d$ in $U_{2e-l}$.  Again $\hilb v {\Delta'} a = \hilb v d a$.
  Therefore $\hilb v d{U_{s_v}}=1$, and then
  Theorem~\ref{thm:dyadic-hilbert-symbol} implies that $d$ is in
  $U_{2e-s_v}$. But by the choice of $k$ the unit $d$ is not in any
  $U_{m}$ for any $m\ge 2e-l+2$, and therefore $2e-s_v \le 2e-l+1$,
  i.e.~$s_v\ge l-1$. Since $U_{l-1}U_{(l-1)/2}^2=U_{l-2}$ (see
  Section~\ref{sec:Hilbert-symbol}, Lemma~\ref{lem:units-mod-squares}
  below) we conclude $s_v\le l$, which was to be shown.  This proves
  the theorem.
\end{proof}

\section{Zeta functions associated to discriminants}
\label{sec:zeta-functions}

\subsection{A reciprocity law}

As in the preceding section $K$ denotes an arbitrary number field.
For any discriminant $\Delta$ of~$K$, we set\footnote{Note that
  $\zeta(\Delta,s)$ depends also on $K$ since $\Delta$ might be a
  discriminant in various fields. However, since we fixed $K$ once and
  for all we suppress this dependence in the notation.}
\begin{equation}
  \label{eq:N-zeta-definition}
  \zeta(\Delta,s) := \sum_{\id a} \frac {N_{\Delta}(\id a)}{\N(\id
    a)^{s}}
\end{equation}
where
\begin{equation*}
  N_{\Delta}(\id a)=\card {\left\{b\bmod
      2\id a: \Delta\equiv b^2\bmod 4\id a\right\}}
  ,
\end{equation*}
and where the sum is over all integral ideals $\id a$ of~$K$.
  
The series $\zeta(\Delta,s)$ converges for $\Re(s)>1$. By the Chinese
remainder theorem $N_{\Delta}(\id a)$ is multiplicative in~$\id a)$
and possesses hence an Euler product. For a prime ideal power
$\id p^k$ which is relatively prime to~$\Delta$ we have
$N_{\Delta}(\id p^k)=1+\leg \Delta{\id p}$. Accordingly,
$\zeta(\Delta,s)$ equals up to a finite Euler product the product of
$\zeta_K(s)/\zeta_K(2k)$ with
$L\left(\leg \Delta*,s\right)=\sum_{\id a} \leg \Delta{\id a}\N(\id
a)^{-s}$,
the sum being over all integral ideals of $K$ relatively prime to
$\Delta$.

For explaining the precise connection between $\Delta$ and
$\zeta_K(s)L\left(\leg \Delta*,s\right)/\zeta_K(2k)$, we define a
function $\chi_\Delta$ on the semigroup of all integral ideals $\id a$
by setting
\begin{equation}
  \label{eq:chi-Delta}
  \chi_{\Delta}(\id a)
  :=
  \begin{cases}
    \N(\id g) \, \uleg \Delta {\id a/\id g^2}
    & \text{if $(\id a,\Delta)=\id g^2$ and $\Delta$ is a square mod~$4\id g^2$,}\\
    0 & \text{otherwise} .
  \end{cases}
\end{equation}
Note that, for an integral square $\id g^2\mid\Delta$, the condition
that $\Delta$ is a square mod~$4\id g^2$ is equivalent to
$\id g\mid \id f_\Delta$.  Of course, $\chi_\Delta$ is no longer a
homomorphism, but it remains multiplicative in the sense that
$\chi_\Delta(\id a\id b) = \chi_\Delta(\id a)\chi_\Delta(\id b)$
whenever $\id a$ and $\id b$ are relatively prime. We use $\N(\id g)$
for the (absolute) norm of~$\id g$
(i,e.~$\N(\id g)=\card {\D /\id g}$).

If we associate to a function $\psi$ defined on integral ideals of~$K$
the Dirichlet series
$L(\psi,s):=\sum_{\id a} \psi(\id a)\N(\id a)^{-s}$ (the sum being
over all integral ideals of~$K$), then it is a bit tedious but
straight-forward to rewrite the definition of~$\chi_\Delta$ in the
form
\begin{equation}
  \label{eq:chi-Delta-as-L-series}
  L(\chi_\Delta,s)
  =
  L\left(\uleg \Delta*,s\right)
  \sum_{\id t \mid \id f_\Delta}
  \frac {\mu(\id t)\uleg \Delta {\id t}}{\N(\id t)^s}\,\sigma_{1-2s}(\id f_\Delta/\id t)
  ,
\end{equation}
where we set $\uleg \Delta {\id t}=0$ if $\id t$ is not relatively
prime to $\Delta/\id f_\Delta^2$, where $\mu(\id a)$ is the M\"obius
$\mu$-function of~$K$, and where
$\sigma_s(\id f)=\sum_{\id t\id f}\N(\id t)^s$. We leave this identity
as an exercise; the interested reader may as well look up
its simple proof in~\cite[Lemma~6.1]{Boylan-Jacobi-Eisenstein}.

\begin{Theorem}
  \label{thm:counting-formula}
  For any discriminant $\Delta$ in $K$ and integral ideal~$\id
  a$ of $K$, one has
  \begin{equation}
    \label{eq:key-identity}
    \card {\left\{x\in\D/2\id a : x^2\equiv \Delta\bmod 4\id a\right\}}
    =
    \sum_{
      \begin{subarray}c
        \id b\mid\id a\\
        \id a/\id b \text{ squarefree}
      \end{subarray}
    }
    \chi_\Delta(\id b)
    .
  \end{equation}
  (The sum is over all integral ideals diving $\id
  a$ and such that $\id a/\id b$ is squarefree.) In other words,
  \begin{equation}
    \label{eq:the-identity}
    \zeta(\Delta,s)
    =
    \frac {\zeta_K(s)}{\zeta_K(2s)}L(\chi_\Delta,s)
    ,
  \end{equation}
  where $\zeta(\Delta,s)$ is the zeta function~\eqref{eq:N-zeta-definition}.
\end{Theorem}

\begin{Remark}
  This formula~\eqref{eq:the-identity} is a classical fact if $K=\Q$
  and $\Delta$ is a negative fundamental discriminant of~$\Q$ (so that
  $\chi_{\Delta}$ is simply the character of~$\Q(\sqrt \Delta)$).
  In~\cite[Prop.~3]{Zagier} it is shown that the formula holds still
  true if $K=\Q$ and for all rational discriminants~$\Delta$,
  where~$L(\chi_\Delta,s)$ is defined
  by~\eqref{eq:chi-Delta-as-L-series} (with $k=\Q$). The theorem
  therefore shows that is holds true in any number field with the
  appropriate definition~\eqref{eq:chi-Delta} of~$\chi_\Delta$.
\end{Remark}

\subsection{Proof of Theorem~\ref{thm:counting-formula}}

\begin{proof}
  Denote the left hand side of the claimed identity by
  $S_\Delta(\id a)$. Note that $S_\Delta(\id a)$ is multiplicative
  in~$\id a$.  Indeed, the $\Delta\equiv s^2\bmod 4$ has a unique
  solution~$s$ mod~$2$. Hence $S_\Delta(\id a)$ equals the number of
  solutions mod~$\id a 2_{\id a}$ of
  $x^2\equiv\Delta\bmod \id a 2_{\id a}^2$, where $2_{\id a}$ is the
  product of all prime powers $\id p^e$ exactly dividing~$2$ where
  $\id p\mid a$. Using this description of $S_\Delta(\id a)$ the
  claimed multiplicativity follows now from the Chinese remainder
  theorem.

  It suffices therefore to prove~\eqref{eq:key-identity} for prime
  ideal powers~$\id p^k$ ($k\ge 1$), i.e.~it suffices to prove, for
  $k\ge 1$,
  \begin{equation}
    \label{eq:key-identity-local}
    S_\Delta(\id p^k)
    =
    \chi_\Delta(\id p^k)+\chi_\Delta(\id p^{k-1})
    .
  \end{equation}
  Moreover, we can replace $\D$ and $\id p$ in the definition of
  $S_\Delta(\id p^k)$ by the localization $\D_{\id p}$ and the
  principal ideal $\widehat{\id p}=\pi\D_{\id p}$, where $\pi$ is a
  uniformizing parameter for the maximal ideal of~$\D_{\id p}$.

  Assume first of all $\id p \nmid \Delta$. The right hand side
  of~\eqref{eq:key-identity-local} equals then $1+\chi_\Delta(\id p)$,
  which is~$2$ or~$0$ according as $\Delta$ is a square mod~$4\id p$
  or not. This proves~\eqref{eq:key-identity-local} for $k=1$.  But
  then~\eqref{eq:key-identity-local} is also true for all $k\ge 1$
  since $S_\Delta(\id p^k)=S_\Delta(\id p)$. For this it suffices to
  show that, for $k\ge 1$, the canonical reduction map
  $\rho:S_\Delta(\id p^{k+1})\rightarrow S_\Delta(\id p^{k})$ is a
  bijection.  Indeed, let $x$ be in~$S_\Delta(\id p^{k})$, and let $y$
  in $x+2\widehat{\id p}^k$, say $y=x+2\pi^kt$ for some $t$
  in~$\D_{\id p}$. The congruence
  \begin{equation*}
    \left(x+2\pi^k t\right)^2\equiv \Delta\bmod 4\pi^{k+1}   
  \end{equation*}
  is equivalent to $xt\equiv \frac {\Delta - x^2}{4\pi^k} \bmod \pi$,
  which has exactly one solution mod~$\pi$ (since $\pi\nmid x$).

  Next, suppose that $\id p^l$ for some $l\ge 1$ is the exact power
  of~$\id p$ dividing~$\Delta$, and let~$\id p^e$ be the exact
  $\id p$-power dividing~$2$.
  
  For ${2e+k} \le l$, the congruence $x^2\equiv \Delta\bmod 4\id p^k$
  is equivalent to $\id p^{e+\lceil k/2\rceil} \mid x$,
  $x\equiv s\bmod 2$ (where $s^2\equiv\Delta\bmod4$), and hence has
  $\N(\id p)^{\lfloor k/2\rfloor}$ solutions mod~$2\id p^k$.  But
  $\N(\id p)^{\lfloor k/2\rfloor}$ equals also the right hand side
  of~\eqref{eq:key-identity} since one of the terms is zero and the
  other one equals $\N(\id p)^{\lfloor k/2\rfloor}$. (For the
  verification note that $2e\le l-k$, so that
  $\Delta/\id p^{2\lfloor k/2\rfloor}$ is divisible by~$\id p^{2e}$,
  and hence still a square modulo~$4$.

  For $2e+k>l$ and odd~$l$, the congruence
  $x^2\equiv \Delta\bmod 4\id p^k$ has no solution (since a
  square~$x^2$ cannot have the odd $\id p$-power $\id p^l$ as exact
  divisor). But the right hand side of~\eqref{eq:key-identity-local}
  is also zero since, for any $k'$ with $2e+k'\ge l$, either the gcd
  of $\Delta$ and~$\id p^{k'}$ equals~$\id p^l$ (which is not a
  square), or else it equals $\id p^{k'}$, where $\pi^{k'}$ is not a
  square, or where $2e>l-k'$ and hence $\Delta/\pi^{k'}$ is not a
  square mod~$4$).

  Finally, let $2e+k>l$ and $l$ be even.  Then the congruence
  $x^2\equiv \Delta\bmod 4\id p^k$ is equivalent to
  $x\equiv \pi^{l/2}y\bmod 2\pi^k$ and
  $y^2\equiv \Delta/\pi^l\bmod 4\pi^{k-l}$. In other words,
  \begin{equation*}
    S_\Delta(\id p^k)
    =\card {\left\{
        y\in \D_{\id p}/2\pi^{k-l/2}:
        y^2\equiv \Delta'\bmod 4\pi^{k-l}
      \right\}}
    ,
  \end{equation*}
  where $\Delta'=\Delta/\pi^l$.

  Suppose $\Delta'$ is a square mod~$4$. Then, for $k>l$, we have
  \begin{equation*}
    S_\Delta(\id p^k)=\N(\id p)^{l/2}S_{\Delta'}(\id p^{k-l})
    =
    \N(\id p)^{l/2}\left(\chi_{\Delta'}(\id p^{k-l})+\chi_{\Delta'}(\id p^{k-l-1})\right)
    ,
  \end{equation*}
  where the second identity follows from the already proven validity
  of~\eqref{eq:key-identity-local} for squares $\Delta$ mod~$4$ not
  divisible by~$\id p$. But the right hand side of the last identity
  equals the right hand side of~\eqref{eq:key-identity-local}.
  If $k\le l$ then $y^2\equiv \Delta'\bmod 4\pi^{k-l}$ is equivalent
  to $y\equiv s\bmod 2\pi^{\lceil \frac {k-l}2\rceil}$. Hence
  $S_\Delta(\id \pi^k)=\N(\id p)^{\lfloor k/2\rfloor}$, which again
  proves~\eqref{eq:key-identity-local}.

  Suppose now that $\Delta'$ is not a square mod~$4$. Then, for
  $k\ge l$ both sides of~\eqref{eq:key-identity-local} equal~$0$.  If
  $l>k$ (and $k>l-2e$, so that in particular, $e\ge 1$) then
  $y^2\equiv \Delta'\bmod 4\pi^{k-l}$ might have a solution or
  not. Let $\delta=\delta(\Delta')$ be the largest integer
  $1\le \delta\le 2e-1$ such that $\Delta'$ is a square
  mod~$\pi^{\delta}$. Note that $\delta$ must be odd (see
  Lemma~\ref{lem:units-mod-squares} below).  Note also that $\Delta'$
  is a square mod~$\id p$ (since squaring is the Frobenius isomorphism
  in a field of characteristic~$2$). Note also that the congruence
  $\Delta'\equiv x^2\bmod \pi^{t}$ with $t<2e$ determines $x$ uniquely
  mod~$\pi^{\lceil t/2\rceil}$. We therefore find
  $S_\Delta(\id p^k)=\N(\id p)^{\lfloor k/2\rfloor}$ for
  $1\le 2e+k-l \le \delta$, and $S_\Delta(\id p^k)=0$ for
  $\delta+l-2e<k<l$. Again this equals the right hand side
  of~\eqref{eq:key-identity-local} since, for $k'=k$ or $k'=k-1$, we
  have $(\Delta,\id p^{k'})=\id p^{k'}$ and exactly one of these $k'$
  is even, and for this~$k'$ we have
  $\chi_\Delta(\id p^{k'})=\N(\id p)^{\lfloor k/2\rfloor}$ or
  $\chi_\Delta(\id p^{k'})=0$ according as $\Delta/\pi^{k'}$ is a
  square mod~$4$, or not. But $\Delta/\id p^{k'}$ being a square
  mod~$4$ is equivalent to $\Delta'$ being a square mod~$2e+k'-l$,
  i.e.~$2e+k'-l\le \delta$.  Since $\delta$ is odd, the latter is
  equivalent to $2e+k-l\le \delta$.  This completes the proof of
  Theorem~\ref{thm:counting-formula}.
\end{proof}

\subsection{The function $\zeta(\Delta,s)$ in the non-relative case}

If $K$ equals the field of rational numbers it can be quickly verified
that, for any pair $(m,n)$ of relatively prime integers, one has for
the series~\eqref{eq:N-zeta-definition} the identity
\begin{equation*}
  \zeta(\Delta,s)
  =
  \sum_{
    \begin{subarray}c
      f\in \mathcal{F}(\Delta)/\SL [\Z]_{(m,n)}
      \\
      f(m,n)>0
    \end{subarray}
  }
  \frac 1{f(m,n)^s}
  ,
\end{equation*}
where the sum is over a complete system of representatives for the set
$\mathcal{F}(\Delta)$ of integral binary quadratic forms
$f=[a,b,c]=ax^2+bxy+cy^2$ satisfying $b^2-4ac = \Delta$ and $f(m,n)>0$
modulo the stabilizer of $(m,n)$ in $\SL [\Z]$. The set of orbits is
defined via the usual action $(f,A)\mapsto f\left(A(x,y)^t\right)$ of
$\SL [\Z]$ on binary forms. The stated identity is obvious for
$(m,n)=(1,0)$: Here we have $\SL [\Z]_{(m,n)}=\langle \mat
1101\rangle$, and hence the series on the right of the last identity
becomes $\sum_{[a,b,c]}{a^{-s}}$, the sum being over a set of
representatives $[a,b,c]$ for $\mathcal{F}^+(\Delta)/\langle \mat
1101\rangle$, where the `$+$' indicates the subset of all $[a,b,c]$ in
$\mathcal{F}(\Delta)$ with $a>0$. As representatives one can take the
forms $[a,b,c]$ in $\mathcal{F}^+(\Delta)$ with $0\le b<2a$, and one
recognizes the series $\zeta(\Delta,s)$.

One can write the identity of the last paragraph also in another way.
According to the well-known theory of binary quadratic forms (see
e.g.~\cite[]{}) The application $[a,b,c]\mapsto \id A_{[a,b,c]}:=\Z a
+ \Z\frac {-b+\sqrt \Delta}2$ maps $\mathcal{F}(\Delta)$ to the set of
ideals of the order $R_\Delta:=\Z+\Z\frac {\Delta+\sqrt
  \Delta}2=\Z+\Z\frac {-b+\sqrt \Delta}2$. In fact, it induces a
bijection $\mathcal{F}^+(\Delta)/\langle \mat 1101\rangle\cong
\mathcal{I}(\Delta)/\Q^*$, where $\mathcal{I}(\Delta)$ is the set of
all fractional ideal of~$R_\Delta$. The ideals of $R_\Delta$ in $\id
A_{[a,b,c]}\Q^*$ are $\id A_{[a,b,c]}\Z_{\ge1}$, and hence the
identity of the last paragraph can be stated as
\begin{equation*}
  \zeta(2s)\zeta(\Delta,s)=\sum_{\id A}\frac1{[R_\Delta:\id A]^s}
  ,
\end{equation*}
where the sum on the right is over all ideals of $R_\Delta$. In terms
of $L(\chi_\Delta,s)$ this can also be written (using
Theorem~\ref{thm:counting-formula}) as
\begin{equation*}
  \zeta(s)L(\chi_{\Delta},s) = \sum_{\id A}\frac1{[R_\Delta:\id A]^s}
  .
\end{equation*}
As we shall show in the next section the last two identities
generalize to relative quadratic extensions.

\subsection{Orders and $\D$-lattices of relative quadratic extensions}

We fix again a number field $K$ with ring of integers $\D$. Let $L$ be
a quadratic extension of $K$ and $\DL$ the ring of integers of $L$.
An {\em $\D$-lattice in $L$} is a finitely generated $\D$-submodule of
$L$ containg a $K$-basis of $L$. If $\id L$ is an $\D$-lattice in $L$ we
define the {\em order of $\id L$} to be the set of elements $a$ in $L$
such $a\id L\subseteq \id L$. This order is then an order of $L$ in the usual
sense, i.e.~it is a subring of finite index in $\DL$. Note that
$\DL$ contains $\D$, and that every order of $L$ containing $\D$ is
the order of an $\D$-module (namely the order of itself, viewed as
$\D$-module). If $\Delta$ is a discriminant in $K$ and $b$ in $\D$ a
solution of of $\Delta\equiv b^2\bmod 4$ we set
\begin{equation*}
  \omega_{\Delta,b}:=\frac {b+\sqrt \Delta}2
  .
\end{equation*}
If $L=K(\sqrt \Delta)$ then $\omega_{\Delta,b}$ is an integral element
of~$L$.

\begin{Lemma}
  \label{lem:key-lemma}
  Every $\D$-lattice in $L$ containing $\D$ is of the form
  $\id o + \id a^{-1} \omega_{\Delta,b}$ with an integral ideal $\id a$
  of~$K$, a discriminant $\Delta$ of~$K$, and an integer $b$ of $\D$
  such that $\Delta\equiv b^2\bmod 4$.
\end{Lemma}

\begin{proof}
  Let $\id L$ denote a $\D$-lattice in $L$ containing $\D$.
  Since the $\D$-module $\id L/\D$ is of rank~$1$, it is isomorphic to
  a fractional ideal~$\id b$. Multiplying $\id b$ by a suitable element
  of~$\units K$ we can assume that $\id b^{-1}$ is integral. We then
  have the exact sequence of $\D$-modules
  \begin{equation*}
    0
    \xrightarrow{}\D
    \xrightarrow{\subseteq}\id L
    \xrightarrow{\pi}\id L/\D
    \xrightarrow{f}\id b
    \xrightarrow{}0
    ,
  \end{equation*}
  where $\pi$ is the canonical projection and $f$ and isomorphism.
  Since $\id b$ is projective we can find a section
  $s:\id b\rightarrow \id L$ (i.e.~$f\circ \pi\circ s=1$). It follows
  $\id L=\id o+s(\id b)$. Setting $\omega := s(1)$ we obtain
  $s(\id b)=\id b\omega$ (If $n$ is in $\D$ such that $n\id b$ is
  integral, then $ns(\id b)=s(n\id b)=n\id b s(1)$ since $s$ is a
  $\D$-homomorphism.) But $a\omega^2-b\omega+c=0$ for suitable $a,b,c$
  in~$\D$, and hence $\id b\omega = (\id b/a)\frac {b+\sqrt \Delta}2$
  with $\Delta=b^2-4ac$. The lemma is now obvious.
\end{proof}

\begin{Proposition}
  \label{prop:description-of-orders}
  Let $\Delta$ be a discriminant of~$K$, $L=K(\sqrt \Delta)$ and $b$
  in $\D$ a solution of $\Delta\equiv b^2\bmod 4\id f_{\Delta}^2$.
  \begin{enumerate}
  \item One has $\DL=\D+\id f_\Delta^{-1}\omega_{\Delta,b}$.
  \item Every order of $L$ containg $\D$ is of the form
    \begin{equation*}
      \DL_{\id c}
      := \D +\id c \DL
      = \D + \id c \id f_\Delta^{-1}\omega_{\Delta,b}
    \end{equation*}
    for a unique integral ideal $\id c$ of $K$.
  \end{enumerate}
\end{Proposition}

\begin{proof}[Proof of Proposition~\ref{prop:description-of-orders}]
  For (1) we use the preceding lemma to write
  $\DL=\id o+\id a^{-1}\omega_{\Delta',b'}$ with an integral ideal
  $\id a$, a discriminant $\Delta'$ of~$K$ and integral $b'$ with
  $\Delta'\equiv b'^2\bmod 4$. Since the elements of~$\DL$ are
  integral we conclude $\id a\mid b'$, and
  $\id a^2\mid \frac {{b'}^2-\Delta'}4$. It follows
  $\id a\mid \id f_{\Delta'}$, hence
  $\DL\subseteq \D+\id f_{\Delta'}^{-1}\omega_{\Delta',b'}$, and then even
  equality since the elements of $\id f_{\Delta'}^{-1}\omega_{\Delta',b'}$
  are integral.  Finally one has
  $\id f_{\Delta'}^{-1}\omega_{\Delta',b'}=\id f_{\Delta}^{-1}\omega_{\Delta,b}$
  as is quickly shown on writing $\Delta'=a^2\Delta$ with $a$ in~$K$
  and using $f_{\Delta'}=a\id f_{\Delta}$ (see
  Corollary~\ref{prop:fD-class-invariance}).

  For (2) let $R$ be an order of~$L$ containing~$\D$. Then $R$ is of
  finite index in~$\DL$. Let $\id c$ be the annihilator of
  $\DL/R$. Note that by~(1) $\id c$ equals the ideal of all $a$
  in~$\D$ such that $a\id f_\Delta^{-1}\omega_{\Delta,b}\subseteq R$
  since $\D$ is contained in $R$. We claim $R=\D+\id c \DL$. The
  inclusion~`$\supseteq$' is trivial. Let $a$ be in $R$. Then
  $a=x+y\omega_{\Delta,b}$ for suitable $x$ in $\D$ and $y$ in
  $\id f_\Delta$. But then
  $\D (a-x)=y\D \omega_{\Delta,b}=(y\id
  f_\Delta)f_\Delta^{-1}\omega_{\Delta,b}\in R$,
  which implies $yf_\Delta\subseteq \id c$. Finally, the $\id c$ in
  the representation $R=\D+\id c\DL$ is unique since, for any integral
  ideal~$\id c$ in $K$ the annihilator of $\DL/(\D+\id c\DL)$ equals~$\id c$.
\end{proof}

For a discriminant $\Delta$ of $K$ let $\mathcal Q(\Delta)$ be the
set of all pairs $(\id a,b)$, where $\id a$ is an integral ideal in
$K$, where $b$ is in $\D$ and where $\Delta\equiv b^2\bmod 4\id a$.
For $(\id a,b)$ in $\mathcal Q(\Delta)$ set
\begin{equation*}
  \id I_{\id a,b,\Delta}
  :=
  \id a + \D \omega_{\Delta,b}
  .
\end{equation*}
Suppose $L=K(\sqrt \Delta)$. Then $\id I_{\id a,b,\Delta}$ is clearly an
ideal of $\DL_{\id f_\Delta}$ (as is immediatelty clear since by
noting that the preceding proposition implies
$\DL_{\id f_\Delta}=\D + \D\omega_{\Delta,b}$ for every solution $b$ of $\Delta\equiv b^2\bmod 4$). We let
$I_L(\Delta)$ and $I_K$ denote the monoids of fractional ideals of~$\D_{\id f_{\Delta}}$ and $\D$. (Recall that a fractional ideal of
$\DL_{\id f_{\Delta}}$ is a finitely generated
$\DL_{\id f_{\Delta}}$-submodule $\id A$ of $L$.)

\begin{Theorem}
  \label{thm:zeta-delta-interpretation}
  Let $\Delta$ be a discriminant in $K$ such that $L=K(\sqrt \Delta)$.
  The application $(\id a,b)\mapsto \id I_{\id a,b,\Delta}$ induces a
  bijection
  \begin{equation*}
    \mathcal Q(\Delta)/\mathord{\thicksim} \longrightarrow I_L(\Delta)/I_K
    ,
  \end{equation*}
  where $(\id a,b)\sim(\id a',b')$ if $\id a=\id a'$ and $b\equiv
  b'\bmod 2\id a$.  In particular, one has
  \begin{equation}
    \label{eq:order-zeta-function}
    \sum_{\id A} [\DL_{\id f_\Delta}:\id A]^{-s}
    =
    \zeta_K(s)L(\chi_\Delta,s)
    =
    \zeta_K(2s) \zeta(\Delta,s)
    ,
  \end{equation}
  with $\chi_\Delta$ as in~\eqref{eq:chi-Delta}, and where the sum on
  the left is over all ideals of~$\DL_{\id f_\Delta}$.
\end{Theorem}

\begin{proof}
  It is clear that $\id I_{\id a,b,\Delta}$ depends only on $b$
  modulo~$\id a$ so that the given application induces indeed a map
  $f:\mathcal Q(\Delta)/\mathord{\thicksim} \longrightarrow I_L(\Delta)/I_K$.

  The map $f$ is surjective: By Lemma~\ref{lem:key-lemma} every
  element $I_L(\Delta)/I_K$ is represented by an $\D$-lattice of the
  form $\id A:=\D + \id a^{-1}\omega_{\Delta',b'}$ with an integral
  ideal $\id a$, a discriminant $\Delta'$ and a solution $b'$ in~$\D$
  of $\Delta'\equiv {b'}^2\bmod 4$. Since $\Delta'$ and $\Delta$
  differ by a square in~$\units K$ we can assume (after possibly
  multiplying $\id A$ and $\omega_{\Delta',b'}$ by a suitable integer
  in $K$) that $\Delta'=a^2\Delta$ for some integer $a$ in~$\D$.  From
  $\DL_{\id f_{\Delta}}\id A = \id A$ we deduce (on writing, using
  Proposition~\ref{prop:description-of-orders},
  $\DL_{\id f_{\Delta}}=\D+\D\omega_{\Delta,b}$ with any
  $\Delta\equiv b^2\bmod 4$) that $\omega_{\Delta,b}\in\id A$. But
  $\omega_{\Delta,b}=\frac 1a \omega_{a^2\Delta,b'} + \frac
  {b-b'/a}2$,
  so that $\omega_{\Delta,b}\in\id A$ implies $a\mid \id a,b$ and
  $b\equiv b'/a\bmod 2$. Writing $\id a$ for $\id a/a$, replacing $b$
  by $b'/a$, we find $\id A=\id o+\id a^{-1}\omega_{\Delta,b}$. Again
  from $\DL_{\id f_{\Delta}}\id A = \id A$ we deduce
  $\omega_{\Delta,b}^2\omega \in \id A$, which finally implies
  $\frac {b^2-\Delta}4\in \D$.

  The map $f$ is injective: Every class in $I_L(\Delta)/I_K$ contains
  exactly one element of the form $\id A := \id I_{\id
    a,b,\Delta}$.
  Suppose $\id A':=\id I_{\id a',b',\Delta}$ is in the same class as
  $\id A$, say, $\id c\id A'=\id A$ for some fractional ideal $\id c$
  in~$K$. Then $\id c\id a'=\id c\id A'\cap K = \id A\cap K=\id a$,
  and hence
  $\id a +\id c\omega_{\Delta,b'}=\id a +\D\omega_{\Delta,b}$.  But
  then $a+c\omega_{\Delta,b'}=\omega_{\Delta,b}$ for suitable
  $a$~in$\id a$ and $c$ in~$\id c$, which in turn implies $c=1$,
  whence $\id c=\D$, and $b'\equiv b\bmod 2\id a$.

  For proving~\eqref{eq:order-zeta-function} we note that
  $\id b\mapsto \id b\id I_{\id a, b\Delta}$ defines a bijection
  between the set of integral ideals of~$K$ and the subset of
  $\DL_{\id f_\Delta}$-ideals in the class of $I_{\id a, b\Delta}$ in
  $I_L(\Delta)/I_K$. Moreover,
  $[\DL_{\id f_\Delta}:\id b \id I_{\id a,b,\Delta}]=\N(\id b)^2\n(\id
  a)$.
  The left hand side of~\eqref{eq:order-zeta-function} equals
  therefore
  \begin{equation*}
    \sum_{\id b}\sum_{(\id a,b)\in \mathcal{Q}(\Delta)/\equiv}
    \N(\id b)^{-2s}\N(\id a)^{-s}
    ,
  \end{equation*}
  and inserting~\eqref{eq:the-identity} we recognize the first of the
  claimed identities. The second one is merely a restatement
  of~\eqref{eq:the-identity} of Theorem~\ref{thm:counting-formula}.
\end{proof}

\subsection{Analytic properties of $L(\chi_{\Delta},s)$}


From its definition it is clear that the $L$-series $L(\chi_\Delta,s)$
converges absolutely for $\Re(s)>1$.

\begin{Theorem}
  \label{thm:analytic-properties}
  Let $\Delta$ be a discriminant in~$K$.  The function
  $L(\chi_\Delta,s)$ in~\eqref{eq:chi-Delta-as-L-series} has the
  following properties.
  \begin{enumerate}
  \item It can be analytically continued to $\C\setminus\{1\}$.
  \item It is an entire function if $\Delta$ is not a square in $K$.
  \item It has a simple pole at~$s=1$ with residue
    \begin{equation*}
      \kappa
      =
      \frac {2^{r_1}(2\pi)^{r_2}}{w_K|D_K|^{1/2}}h_KR_K
    \end{equation*}
    if $\Delta$ is a square in~$K$.
  \item It satisfies the functional equation
    \begin{equation*}
      L^*(\chi_\Delta,s):=\gamma(s)L(\chi_\Delta,s) =
      L^*(\chi_\Delta,1-s)
      ,
    \end{equation*}
    where
    \begin{equation*}
      \phantom{XXX}
      \gamma(s)
      =
      \N(\Delta \dif_K)^{\frac s2}
      \Gamma_\R(s)^{\frac {n+a}2}
      \Gamma_\R(1+s)^{\frac {n-a}2}
      =
      \N(\Delta \dif_K)^{\frac s2}
      \Gamma_\R( s)^{a}
      \Gamma_\C(s)^{\frac {n-a}2}
      .
    \end{equation*}
    Here $a=\sum_\sigma \sym{sign}(\sigma(\Delta))$ with $\sigma$
    running over all real embeddings of~$K$, and
    $\Gamma_\R(s)=\pi^{-s/2}\Gamma(s/2)$ and
    $\Gamma_\C(s)=2(2\pi)^{-s}\Gamma(s)$.
  \end{enumerate}
  Here $n$, $D_K$, $\dif_K$, $h_K$, $R_K$, $w_K$ denote the degree,
  discriminant, different, class number, regulator and number of roots
  of unity of $K$.
\end{Theorem}

\begin{proof}
  We have $L(\uleg {\Delta}*,s)=\zeta_L(s)/\zeta_K(s)$, where
  $L=K(\sqrt\Delta)$ if $\Delta$ is not a square in~$K$, and where
  $\zeta_L(s)=\zeta_K(s)^2$ otherwise. Using this we can write
  \begin{equation}
    \label{eq:Lchi-identity}
    L(\chi_\Delta,s)=\frac {\zeta_L(s)}{\zeta_K(s)}F(s),
  \end{equation}
  where $F(s)$ is the finite Euler product on the right
  of~~\eqref{eq:chi-Delta-as-L-series}.

  It is a classical fact that the Dedekind zeta function of a number
  field can be analytically continued to $\C\setminus\{1\}$, has a
  simple pole at~$s=1$ and satisfies a functional equation
  $s\mapsto 1-s$. For $\Delta$ not a square, $\zeta_L(s)/\zeta_K(s)$
  is an entire function since $L/K$ is Galois (cf.~\cite[]{}). From
  these facts and the preceding identity for $L(\chi_\Delta,s)$ the
  properties~(1) to (4) become obvious. For the formula for the
  residue we note $F(1)=1$, so that the and the residue of
  $L(\chi_\Delta,s)$ at $s=1$ equals the residue $\kappa$ of
  $\zeta_K(s)$ at $s=1$, which by a well-known formula (see
  e.g.~\cite[Chap.~VII, Cor.~(5.11)]{Neukirch}) is given by
  $\rho=\frac {2^{r_1}(2\pi)^{r_2}}{w_K|D_K|^{1/2}}h_KR_K$.

  For the formula for $\gamma(s)$ we note that
  $\N(\id f_{\Delta})^sF(s)$ is invariant under $s\mapsto 1-s$, that
  $\zeta_K^*(s):=|D_K|^{s/2}\Gamma_\R^{r_1}(s)\Gamma_\C^{r_2}(s)=\zeta(1-s)$,
  where $r_1$ and $r_2$ are the real and complex places of $K$, and
  that a similar formula holds for $\zeta_L(s)$ if $\Delta$ is not a
  square. Therefore~(4) holds with
  \begin{multline*}
    \gamma(s)
    =
    \N(\id f_{\Delta})^{s}
    \frac {|D_L|^{s/2}}{|D_K|^{s/2}}\Gamma_\R(s)^{R_1-r_1}\Gamma_\C(s)^{R_2-r_2}
    \\
    =
    \N(\Delta\dif_K)^{s/2}
    \Gamma_\R(s)^{R_1+R_2-r_1-r_2}\Gamma_\R(s+1)^{R_2-r_2}
    ,
  \end{multline*}
  where $R_1$, $r_1$ are the real places and and $R_2$, $r_2$ are the
  complex places of $L$ and~$K$. For the second identity we used
  $D_L=D_K^2\N(D_{L/K})$ (see e.g.~\cite[Ch.~III,
  Cor.~(2.10)]{Neukirch}), $D_{L/K}=\Delta/\id f_{\Delta}^2$
  (Theorem~\ref{eq:discriminant-formula}), and the duplication formula
  $L_\C(s)=L_\R(s)L_\R(s+1)$. (These identities are literally true if
  $\Delta$ is not a square in~$K$, but hold also true with the
  convention $D_L=D_K^2$ and $R_j=2r_j$ for squares $\Delta$ in
  $K$). If $\Delta$ is not a square then every complex place of $K$
  splits into two complex places of $L$, and every real embedding
  $\sigma$ can be continued to two real places of $L$ if
  $\sigma(\Delta)>0$, and can be continued to a pair of complex places
  $\tau$ and $\overline\tau$ in $L$ if
  $\sigma(\Delta)<0$. Accordingly, if we use $r_1^+$ and $r_1^-$ for
  the real places $\sigma$ of $K$ with $\sigma(\Delta)>0$
  $\sigma(\Delta)<0$, respectively, then $R_1=2r_1^+$ and
  $R_2=r_1^-+2r_2$. In other words, $R_2-r_2=\frac {n-a}2$ and
  $R_1+R_2-r_1-r_2=\frac {n+a}2$, and we recognize the claimed formula
  for~$\gamma(s)$.
\end{proof}

\subsection{Special values}

There is another situation where relative quadratic extension continue
smoothly properties of quadratic number fields. This is the case of a
totally real field $K$ and a totally imaginary extension.

\begin{Theorem}
  \label{thm:special values}
  Let $K$ be a totally real number field of degree $n$ and $\Delta$ a totally negative discriminant
  of~$K$. Then
  \begin{equation*}
    H(\Delta) = H(\Delta,K) := L(\chi_\Delta,0)
  \end{equation*}
  (with $\chi_\Delta$ as defined in~\eqref{eq:chi-Delta}) is a
  rational number. More precisely, one has
  \begin{equation}
    \label{eq:HD-formula}
    H(\Delta)
    = \frac {2^{n-1}}{\gamma_L}
    \frac {h_{L}/h_K}{w_{L}/2}
    \sum_{\id f\mid\id f_{\Delta}} \N(\id f)
    \prod_{\id p\mid \id f}
    \left(1-\frac {\uleg \Delta {\id p}}{\N(\id p)}\right)
    .
  \end{equation}
  Here $h_{L}$ and $h_K$ denote the class numbers of
  $L=K(\sqrt\Delta)$ and~$K$, and we use
  $\gamma_L=[\units {\DL_L}:W_{L}\units{\D}]$ with
  $\units {\DL_L}$ and $W_{L}$ denoting the group of
  units and roots of unity in $L$, and
  $w_{L}=\card {W_{L}}$.
\end{Theorem}

\begin{Remark}
  The numbers $H(\Delta,\Q)$ are the Hurwitz class numbers, i.e.~the
  number of $\SL [\Z]$-equivalence classes of binary positive definite
  integral quadratic forms of discriminant $\Delta$, where forms
  equivalent to a multiple of $x^2+y^2$ or $x^2+xu+y^2$ are counted
  with weight $\frac 12$ and $\frac13$, respectively. For $K=\Q$ the
  identity~\eqref{eq:HD-formula} is classical.

  For the computation of $H(\Delta)$ the following rearks are useful:
  \begin{enumerate}
  \item Since $K$ is totally real and $L:=K(\sqrt \Delta)$ totally
    imaginary, the rank of the unit groups of $K$ and $L$ are equal
    and in fact equal to $n-1$ by Dirichlet's unit theorem. In
    particular, $\gamma_L$ is well-defined.
  \item As observed in~\cite[\S2]{Remak} one has $\gamma_L\le 2$.
    Indeed, $W_L\units {\D}$ equals the kernel of the homomorphism
    $\xi:\units{\DL_L} \rightarrow W_L/W_L^2\cong\{\pm 1\}$,
    $u\mapsto (\overline u/u) W_L^2$.
  \item $W_{L}\not=\{\pm 1\}$ can occur only for finitely many
    $\Delta$ modulo~${\units K}^2$. Indeed, if
    $W_{L}=\langle z\rangle$ with a primitive $k$th root of unity $z$,
    then $\varphi(k)\mid 2n$ and, for $k>2$, one has $L=K(z)$. Note
    that $k$ must then also be even (since $W_L$ contains~$-1$), and
    $K$ must contain the totally real subfield $Q(z+\overline z)$ of
    the $k$th cyclotomic field.
  \item One has $\units {\DL_L} = \units {\D}$ for all but finitely
    many $\Delta$ modulo~${\units K}^2$: If
    $\units {\DL_L} \not= \units {\D}$ and $W_L=\{\pm 1\}$ then
    $\gamma_L=[\units {\DL_L}:\units {\D}]=2$, and hence $L=K(z)$ with
    a non-real $z$ in~$\units {\DL_L}$. But then $\overline z=-z$,
    i.e.~$-z^2$ is a totally positive unit in~$K$. Thus $L$ belongs to
    the (finite) image of the application which associates to each
    class $u{\units\D}^2$ in $\units\D/{\units\D}^2$ containing only
    totally positive units the field $L=K(\sqrt {-u})$.
  \item Assume that every totally positive unit in $K$ is
    in~${\units \D}^2$. Then one has $\units {\DL_L}\not = \units {\D}$
    if and only if $W_L\not=\{\pm1\}$.  Namely, if $\units {\D}$ is a
    proper subgroup of $\units {\DL_L}$ then there exists a unit $z$
    in~$L$ and a $u\not=1$ in $W_L$ such that $\overline z = zu$. Then
    either $u\not=-1$ (and hence $w_L>2$) or $z=yi$ for a unit $y$
    in~$\units {\D}$ (and hence $i\in W_L$).  In particular,
    $\gamma_L = 2$ at most if $w_L>2$.
  \item The quotient $h_{L}/h_K$ ($L=K(\sqrt\Delta)$) is an
    integer. This follows from the fact that the norm map from the
    class group of~$L$ to the class group of~$K$ is surjective. (This
    follows by translating this map via class field theory to a map
    $\sym{Gal}(\widetilde L/L)\rightarrow \sym{Gal}(\widetilde K/K)$,
    where $\widetilde L\supseteq \widetilde K$ are the Hilbert class
    fields of $L$ and $K$, and noticing that this map is essentially
    the restriction map; for the latter one uses that the infinite
    places of~$K$ ramify in the totally imaginary extension~$L$, so
    that $\widetilde K\cap L = K$.) Therefore $h_{L}/h_K$ equals the
    cardinality of the kernel of the norm map, i.e.~the cardinality of
    the relative class group of $L/K$.
  \end{enumerate}
\end{Remark}

\begin{proof}[Proof of Theorem~\ref{thm:special values}]
  We note that the sum on the right hand side of the claimed formula
  for $H(\Delta)$ equals $F(0)$ with $F$ as
  in~\eqref{eq:Lchi-identity}. Moreover,
  $\lim_{s\to 0}s^{-r}\zeta_K(s)=h_KR_K/w_k$ and, setting
  $L=K(\sqrt \Delta)$, $\lim_{s\to 0}s^{-R}\zeta_L(s)=h_LR_L/w_L$,
  where $r=r_1+r_2-1$, $R=R_1+R_2-1$, and $R_L$, $R_K$ and $w_L$,
  $w_K$ denote the regulators and the number of roots of unities in
  $K$ and $L$ (see~\cite[???]{Neukirch}). Since by assumption $K$ is
  totally real and $L$ is a totally complex extension of~$K$ we have
  $r=R$ and $R_Lw_K/R_Kw_L=2^{n-1}/[\units{\DL_L}:\units {\D}]$.  But
  $[\units{\DL_L}:\units {\D}]=\gamma_L\cdot [W_L:W_K]$ and
  $W_K=\{\pm 1\}$.  The formula for~$H(\Delta)$ is now obvious.
\end{proof}

\section{Appendix: Tables}
\label{sec:tables}

In this appendix the reader finds tables for the numbers $H(\Delta,K)$
of Theorem~\ref{thm:special values} for the three fields
$\Q(\sqrt 5)$, $\Q(\sqrt {10})$ and $\Q[x]/(x^3 - x^2 - 9x + 10)$. The
field $\Q(\sqrt 5)$ has class number~$1$, whereas the other two fields
are the fields having smallest discriminant among all totally real
fields of degree $2$ and $3$ whose class number is greater than~$1$;
in fact the class number is $2$ in both cases.

A fourth table lists for certain fields $K$ all discriminant classes
$\Delta$ in $K$ such that $K(\sqrt \Delta)$ is unramified at all
finite places of~$K$. By Theorem~\ref{thm:discriminant-formula} the
latter is equivalent to $\Delta/\id f_\Delta^2=1$.  But if
$\Delta/\id f_\Delta^2=1$, then $\uleg \Delta*$ has conductor~$1$ and
is trivial on all principal ideals generated by a totally positive
element (see Theorem~\ref{thm:groessencharakter-property}), so induces
a character of the narrow class group $\Cl^+(K)$ of~$K$ with
$\psi^2=1$.  Class field theory shows that the application
$\Delta\mapsto \uleg \Delta*$ induces in fact a bijection of the (at
finite primes) unramified quadratic extensions of~$K$ 'and the group
$\psi$ of characters of $\Cl^+(K)$ with $\psi^2=1$ (for this recall
from~\S\ref{sec:proof-groessencharakter-property} that $\uleg \Delta*$
equals the Artin reciprocity map of $K(\sqrt \Delta)/K$).  In other
words, the number of (at finite places) unramified extensions
$K(\sqrt \Delta)/K$ is in one to one correspondence with the group
$\Cl^+(K)/\Cl^+(K)^2$ of genus classes of~$K$, whence the number of
such extensions equals the cardinality of the subgroup $\Cl^+(K)[2]$
of elements $C$ in $\Cl^+(K)$ such that $C^2=1$, or, equivalently, the
number of even elementary divisors of the narrow class group. An
unramified extension $K(\sqrt \Delta)/K$ is unramified at the infinite
places too (i.e.~every real place of $K$ splits into two real places
of the extension) if $\Delta$ is totally positive, which, according to
Theorem~\ref{thm:groessencharakter-property}, means that
$\uleg \Delta*$ is trivial on the narrow ideal classes of principal
ideals, i.e.~factors through a character of the class
group~$\Cl(K)$. The number of discriminant classes with totally
positive $\Delta$ such that $\Delta=\id f_\Delta^2$ equals
therefore~$\card{ \Cl(K)[2]}$, i.e.~the number of even elementary
divisors of $\Cl(K)$.

For the fields $K$ of Table~\ref{tab:unit-dicriminants} we chose all
number fields of degree $\le 5$ having smallest discriminant among all
fields with $\card {\Cl^+(K)[2]}\in\{2,4,8\}$. We used the Bordeaux
tables~\cite{nftables} for finding these fields, and for the
computations of this and the other tables we used~\cite{sagemath}. We
did not find any such $K$ in the range of the Bordeaux tables of
degree $5$ and not totally real, or of degree $6$ or $7$. All
computations were done using~\cite{sagemath}.

For the computation of the numbers $H\big(\Delta,K\big)$ for
$K=\Q(\sqrt {5})$ in Table~\ref{tab:x2-5} we note the following. Since
the class number of~$K$ is one every discriminant in~$K$ decomposes as
$\Delta_0a^2$, where $\Delta_0$ is a fundamental discriminant and $a$
an integer in~$K$. The field $K$ is the totally real subfield of the
fifth cyclotomic field $K(\sqrt \Delta_5)$ where
$\Delta_5={\frac {\sqrt 5 -5}2}$.  Therefore, for a given totally
negative fundamental discriminant~$\Delta$, the field
$L=K(\sqrt\Delta)$ contains roots of unity different from~$\pm 1$
exactly if $\Delta\in\{\Delta_5, -3, \Delta_4={2\sqrt 5-6}\}$ up to
multiplication by a square of a unit in~$K$. Since the fundamental
unit of~$K$ has norm~$-1$ every totally positive unit in~$K$ is a
square in~$K$. Accordingly, the group of units of $L=K(\sqrt\Delta)$
and of~$K$ differ at most if $\Delta$ is one of the described three
kinds of discriminants.(see Remark~(5) after Theorem~\ref{thm:special
  values}). For these $\Delta$ one has $\gamma_L=1$.

Table~\ref{tab:x2-10} lists $H\big(\Delta,K\big)$ for
$K=\Q(\sqrt {10})$. This is the totally real quadratic field with
smallest discriminant having class number greater~$1$, in fact, equal
to~$2$. The field $K$ is not the totally real subfield of any
cyclotomic field. Therefore, for a given totally negative $\Delta$,
the field $L=K(\sqrt\Delta)$ contains roots of unity different
from~$\pm 1$ exactly if $L=\Q(\sqrt {10},\sqrt {-3})$ or
$L=\Q(\sqrt {10},\sqrt {-4})$, i.e.~if and only if
$\Delta\in\{-3,-4\}$ (up to multiplication by a square in~$K$). Since
every totally positive unit in~$K$ is a square in~$K$, the group of
units of~$L$ and of~$K$ differ at most if $\Delta\in -3{\units K}^2$
or $\Delta\in -4{\units K}^2$ (see Remark~(5) after
Theorem~\ref{thm:special values}). For these $\Delta$ one has
$\gamma_L=1$.

Finally, Table~\ref{tab:x3-} lists $H\big(\Delta,K\big)$ for
$K=\Q[x]/(x^3-x^2-9x+10)$. This is the totally real cubic field with
smallest discriminant having class number greater than $1$. The
discriminant is $D_K=19 \cdot 103$ and the class number is
$h_K=2$. For a given totally negative $\Delta$, the field
$L=K(\sqrt\Delta)$ contains roots of unity different from~$\pm 1$
exactly if and only if $\Delta\in\{-3,-4\}$ (up to multiplication by a
square in~$K$) as follows from Remark~(3) after
Theorem~\ref{thm:special values} and the fact that $K$ itself is not
the subfield of a cyclotomic field (e.g.~since $K/\Q$ is not
galois). We checked that in both cases $\gamma_L=1$.

If $L=K(\sqrt\Delta)$ has only $\pm1$ as roots of unity, then
$\gamma_L\not\not =1$ (in fact, $\gamma_L\not\not =2$) if and only if
$L=K(\sqrt u)$ for a totally negative unit $u$ in $K$ (see Remark~(4)
after Theorem~\ref{thm:special values}). But modulo~${\units\D}^2$ the
field $K$ contains only one totally negative unit, which is not in
$-1\cdot {\units\D}^*$, namely $a-3$.

\begin{table}[th]\begin{adjustwidth}{-.5cm}{-.5cm}
  \centering
  \caption{The numbers $H\big(\Delta,\Q(\sqrt {5})\big)$ for all
    totally negative discriminants~$\Delta$ in~$\Q(\sqrt {5})$
    modulo~${\units \D}^2$ with $\N(\Delta) \le 500$.}
  \renewcommand {\frac}[2] {\sfrac {#1}{#2}}
  \footnotesize
  \newcolumntype{R}{>{$}r<{$}}

\begin{tabular}{@{} *{4}{R} | *{4}{R} @{}}\boldsymbol{\N(\Delta)}&\boldsymbol{\Delta}&\boldsymbol{\id f_{\Delta}}&\boldsymbol{H(\Delta)}&\boldsymbol{\N(\Delta)}&\boldsymbol{\Delta}&\boldsymbol{\id f_{\Delta}}&\boldsymbol{H(\Delta)}\\ \midrule 5&-\frac{1}{2} \sqrt{5} - \frac{5}{2}&\left(1\right)&\frac{2}{5}&261&-\frac{3}{2} \sqrt{5} - \frac{33}{2}&\left(1\right)&4\\ 9&-3&\left(1\right)&\frac{2}{3}&269&-2 \sqrt{5} - 17&\left(1\right)&2\\ 16&-4&\left(1\right)&1&269&-\frac{11}{2} \sqrt{5} - \frac{41}{2}&\left(1\right)&2\\ 41&\frac{1}{2} \sqrt{5} - \frac{13}{2}&\left(1\right)&2&281&-\frac{7}{2} \sqrt{5} - \frac{37}{2}&\left(1\right)&6\\ 41&-\frac{1}{2} \sqrt{5} - \frac{13}{2}&\left(1\right)&2&281&-4 \sqrt{5} - 19&\left(1\right)&6\\ 49&-7&\left(1\right)&2&304&-6 \sqrt{5} - 22&\left(1\right)&4\\ 61&-2 \sqrt{5} - 9&\left(1\right)&2&304&-2 \sqrt{5} - 18&\left(1\right)&4\\ 61&-\frac{3}{2} \sqrt{5} - \frac{17}{2}&\left(1\right)&2&320&-4 \sqrt{5} - 20&\left(1\right)&4\\ 64&-8&\left(1\right)&2&341&\frac{1}{2} \sqrt{5} - \frac{37}{2}&\left(1\right)&4\\ 80&-2 \sqrt{5} - 10&\left(2\right)&\frac{12}{5}&341&-\frac{1}{2} \sqrt{5} - \frac{37}{2}&\left(1\right)&4\\ 109&\frac{1}{2} \sqrt{5} - \frac{21}{2}&\left(1\right)&2&361&-19&\left(1\right)&8\\ 109&-\frac{1}{2} \sqrt{5} - \frac{21}{2}&\left(1\right)&2&389&-\frac{13}{2} \sqrt{5} - \frac{49}{2}&\left(1\right)&2\\ 121&-11&\left(1\right)&4&389&-\frac{5}{2} \sqrt{5} - \frac{41}{2}&\left(1\right)&2\\ 125&-\frac{5}{2} \sqrt{5} - \frac{25}{2}&\left(-\sqrt{5}\right)&\frac{12}{5}&400&-20&\left(-\sqrt{5}\right)&5\\ 144&-12&\left(2\right)&\frac{8}{3}&405&-\frac{9}{2} \sqrt{5} - \frac{45}{2}&\left(3\right)&\frac{22}{5}\\ 145&-\frac{3}{2} \sqrt{5} - \frac{25}{2}&\left(1\right)&4&409&\frac{3}{2} \sqrt{5} - \frac{41}{2}&\left(1\right)&6\\ 145&-4 \sqrt{5} - 15&\left(1\right)&4&409&-\frac{3}{2} \sqrt{5} - \frac{41}{2}&\left(1\right)&6\\ 149&-2 \sqrt{5} - 13&\left(1\right)&2&421&-2 \sqrt{5} - 21&\left(1\right)&6\\ 149&-\frac{7}{2} \sqrt{5} - \frac{29}{2}&\left(1\right)&2&421&2 \sqrt{5} - 21&\left(1\right)&6\\ 176&-2 \sqrt{5} - 14&\left(1\right)&4&445&-\frac{7}{2} \sqrt{5} - \frac{45}{2}&\left(1\right)&4\\ 176&-4 \sqrt{5} - 16&\left(1\right)&4&445&-6 \sqrt{5} - 25&\left(1\right)&4\\ 209&-\frac{1}{2} \sqrt{5} - \frac{29}{2}&\left(1\right)&4&449&-\frac{11}{2} \sqrt{5} - \frac{49}{2}&\left(1\right)&6\\ 209&\frac{1}{2} \sqrt{5} - \frac{29}{2}&\left(1\right)&4&449&-4 \sqrt{5} - 23&\left(1\right)&6\\ 225&-15&\left(-\sqrt{5}\right)&\frac{14}{3}&464&2 \sqrt{5} - 22&\left(1\right)&4\\ 241&-\frac{9}{2} \sqrt{5} - \frac{37}{2}&\left(1\right)&6&464&-2 \sqrt{5} - 22&\left(1\right)&4\\ 241&-\frac{5}{2} \sqrt{5} - \frac{33}{2}&\left(1\right)&6&496&-6 \sqrt{5} - 26&\left(1\right)&8\\ 256&-16&\left(2\right)&5&496&-4 \sqrt{5} - 24&\left(1\right)&8\\ 261&\frac{3}{2} \sqrt{5} - \frac{33}{2}&\left(1\right)&4&&&&\\ \bottomrule \end{tabular}

  \label{tab:x2-5}
\end{adjustwidth}
\end{table}

\begin{table}[th]\begin{adjustwidth}{-.5cm}{-.5cm}  
  \centering
  \caption{The numbers $H\big(\Delta,\Q(\sqrt {10})\big)$ for all
    totally negative discriminants $\Delta$ in~$\Q(\sqrt {10})$
    modulo~${\units \D}^2$ with $\N(\Delta) \le 500$.}
  \renewcommand {\frac}[2] {\sfrac {#1}{#2}}
  \footnotesize
  \newcolumntype{R}{>{$}r<{$}}

\begin{tabular}{@{} *{4}{R} | *{4}{R} @{}} \toprule\boldsymbol{\N(\Delta)}&\boldsymbol{\Delta}&\boldsymbol{\id f_{\Delta}}&\boldsymbol{H(\Delta)}&\boldsymbol{\N(\Delta)}&\boldsymbol{\Delta}&\boldsymbol{\id f_{\Delta}}&\boldsymbol{H(\Delta)}\\ \midrule 4&-2&\left(1\right)&2&265&-36 \sqrt{10} - 115&\left(1\right)&4\\ 9&-3&\left(1\right)&\frac{4}{3}&265&-6 \sqrt{10} - 25&\left(1\right)&4\\ 16&-4&\left(2, \sqrt{10}\right)&3&321&-34 \sqrt{10} - 109&\left(1\right)&44\\ 36&-4 \sqrt{10} - 14&\left(3, \sqrt{10} + 2\right)&5&321&-8 \sqrt{10} - 31&\left(1\right)&44\\ 36&-6&\left(2, \sqrt{10}\right)&8&324&4 \sqrt{10} - 22&\left(\sqrt{10} - 1\right)&18\\ 36&-8 \sqrt{10} - 26&\left(3, \sqrt{10} + 1\right)&5&324&-12 \sqrt{10} - 42&\left(\sqrt{10} + 2\right)&\frac{64}{3}\\ 41&-2 \sqrt{10} - 9&\left(1\right)&4&324&-18&\left(3\right)&18\\ 41&2 \sqrt{10} - 9&\left(1\right)&4&324&-24 \sqrt{10} - 78&\left(-\sqrt{10} + 2\right)&\frac{64}{3}\\ 49&-7&\left(1\right)&4&324&-4 \sqrt{10} - 22&\left(\sqrt{10} + 1\right)&18\\ 64&-8&\left(2\right)&14&356&-16 \sqrt{10} - 54&\left(2, \sqrt{10}\right)&32\\ 65&-4 \sqrt{10} - 15&\left(1\right)&12&356&-20 \sqrt{10} - 66&\left(2, \sqrt{10}\right)&32\\ 65&-14 \sqrt{10} - 45&\left(1\right)&12&361&-19&\left(1\right)&4\\ 81&-6 \sqrt{10} - 21&\left(3, \sqrt{10} + 2\right)&16&369&-32 \sqrt{10} - 103&\left(3, \sqrt{10} + 2\right)&40\\ 81&-12 \sqrt{10} - 39&\left(3, \sqrt{10} + 1\right)&16&369&4 \sqrt{10} - 23&\left(3, \sqrt{10} + 1\right)&24\\ 89&-8 \sqrt{10} - 27&\left(1\right)&4&369&-4 \sqrt{10} - 23&\left(3, \sqrt{10} + 2\right)&24\\ 89&-10 \sqrt{10} - 33&\left(1\right)&4&369&-10 \sqrt{10} - 37&\left(3, \sqrt{10} + 1\right)&40\\ 96&-20 \sqrt{10} - 64&\left(1\right)&12&384&-40 \sqrt{10} - 128&\left(2, \sqrt{10}\right)&12\\ 96&-4 \sqrt{10} - 16&\left(1\right)&12&384&-8 \sqrt{10} - 32&\left(2, \sqrt{10}\right)&12\\ 100&-10&\left(5, \sqrt{10}\right)&5&400&-20&\left(-\sqrt{10}\right)&42\\ 121&-11&\left(1\right)&12&401&-2 \sqrt{10} - 21&\left(1\right)&36\\ 129&-2 \sqrt{10} - 13&\left(1\right)&20&401&2 \sqrt{10} - 21&\left(1\right)&36\\ 129&2 \sqrt{10} - 13&\left(1\right)&20&409&-30 \sqrt{10} - 97&\left(1\right)&4\\ 144&-8 \sqrt{10} - 28&\left(\sqrt{10} + 2\right)&18&409&-12 \sqrt{10} - 43&\left(1\right)&4\\ 144&-12&\left(2\right)&\frac{40}{3}&416&4 \sqrt{10} - 24&\left(1\right)&20\\ 144&-16 \sqrt{10} - 52&\left(-\sqrt{10} + 2\right)&18&416&-4 \sqrt{10} - 24&\left(1\right)&20\\ 160&-12 \sqrt{10} - 40&\left(1\right)&4&441&-14 \sqrt{10} - 49&\left(3, \sqrt{10} + 2\right)&24\\ 164&-4 \sqrt{10} - 18&\left(2, \sqrt{10}\right)&16&441&-28 \sqrt{10} - 91&\left(3, \sqrt{10} + 1\right)&24\\ 164&4 \sqrt{10} - 18&\left(2, \sqrt{10}\right)&16&465&-16 \sqrt{10} - 55&\left(1\right)&40\\ 196&-14&\left(2, \sqrt{10}\right)&32&465&-26 \sqrt{10} - 85&\left(1\right)&40\\ 201&-4 \sqrt{10} - 19&\left(1\right)&12&481&-18 \sqrt{10} - 61&\left(1\right)&20\\ 201&4 \sqrt{10} - 19&\left(1\right)&12&481&6 \sqrt{10} - 29&\left(1\right)&68\\ 225&-15&\left(5, \sqrt{10}\right)&20&481&-6 \sqrt{10} - 29&\left(1\right)&68\\ 240&-4 \sqrt{10} - 20&\left(1\right)&8&481&-24 \sqrt{10} - 79&\left(1\right)&20\\ 240&4 \sqrt{10} - 20&\left(1\right)&8&484&-22&\left(2, \sqrt{10}\right)&16\\ 249&-2 \sqrt{10} - 17&\left(1\right)&4&489&-20 \sqrt{10} - 67&\left(1\right)&12\\ 249&2 \sqrt{10} - 17&\left(1\right)&4&489&-22 \sqrt{10} - 73&\left(1\right)&12\\ 256&-16&\left(4, 2 \sqrt{10}\right)&15&496&-36 \sqrt{10} - 116&\left(1\right)&40\\ 260&-8 \sqrt{10} - 30&\left(2, \sqrt{10}\right)&16&496&-12 \sqrt{10} - 44&\left(1\right)&40\\ 260&-28 \sqrt{10} - 90&\left(2, \sqrt{10}\right)&16&&&&\\ \bottomrule \end{tabular}

  \label{tab:x2-10}
\end{adjustwidth}
\end{table}

\begin{table}[ht]\begin{adjustwidth}{-.5cm}{-.5cm}  
  \centering
  \caption{ The numbers $H(\Delta,K)$ for
    $K=\Q[x]/(x^3 - x^2 - 9x + 10)$ and for all totally negative
    discriminants $\Delta$ in~$K$ modulo~${\units \D}^2$ with
    $|\N(\Delta)| \le 500$.}
  \renewcommand {\frac}[2] {\sfrac {#1}{#2}}
  \footnotesize
  \newcolumntype{R}{>{$}r<{$}}

\begin{tabular}{@{} *{4}{R} | *{4}{R} @{}} \toprule\boldsymbol{\N(\Delta)}&\boldsymbol{\Delta}&\boldsymbol{\id f_{\Delta}}&\boldsymbol{H(\Delta)}&\boldsymbol{\N(\Delta)}&\boldsymbol{\Delta}&\boldsymbol{\id f_{\Delta}}&\boldsymbol{H(\Delta)}\\ \midrule -475&-3 a^{2} - 2 a + 5&\left(5, a\right)&40&-320&-a^{2} + 3 a - 5&\left(1\right)&40\\ -475&-4 a^{2} + 16 a - 15&\left(5, a\right)&56&-320&-5 a^{2} + 23 a - 25&\left(1\right)&24\\ -432&-3 a^{2} - 3 a + 6&\left(a - 2\right)&64&-304&-a^{2} + 7 a - 13&\left(2, a^{2} + a - 5\right)&48\\ -432&-3 a^{2} + 12 a - 12&\left(a - 2\right)&\frac{160}{3}&-304&-a^{2} - a - 1&\left(2, a^{2} + a - 5\right)&32\\ -412&-3 a^{2} + a + 2&\left(2, a\right)&80&-236&a^{2} + a - 14&\left(2, a\right)&48\\ -412&-7 a^{2} + 28 a - 24&\left(2, a\right)&32&-236&a^{2} + 8 a - 32&\left(2, a\right)&32\\ -404&-2 a^{2} - a + 2&\left(1\right)&32&-104&a^{2} - 12&\left(1\right)&8\\ -404&-3 a^{2} + 13 a - 14&\left(1\right)&32&-104&2 a^{2} + 3 a - 26&\left(1\right)&24\\ -359&-4 a^{2} + 15 a - 13&\left(1\right)&16&-95&a^{2} + 2 a - 15&\left(1\right)&40\\ -359&-4 a^{2} - 4 a + 9&\left(1\right)&80&-95&-a - 5&\left(1\right)&8\\ -352&-7 a^{2} + 29 a - 26&\left(2, a\right)&72&-76&-2 a^{2} + 7 a - 6&\left(2, a\right)&32\\ -352&-2 a^{2} + 3 a - 2&\left(2, a\right)&24&-76&-3 a^{2} - 4 a + 8&\left(2, a\right)&16\\ -347&-3 a^{2} + 2 a + 1&\left(1\right)&64&-64&4 a - 12&\left(1\right)&8\\ -347&-8 a^{2} + 32 a - 27&\left(1\right)&16&-64&-4&\left(1\right)&16\\ -343&7 a - 21&\left(1\right)&16&-27&3 a - 9&\left(1\right)&16\\ -343&-7&\left(1\right)&48&-27&-3&\left(1\right)&\frac{16}{3}\\ \bottomrule \end{tabular}

  \label{tab:x3-}
\end{adjustwidth}
\end{table}

\begin{table}[th]\begin{adjustwidth}{-2cm}{-2cm} 
  \centering
  \caption{The first fields $K=\Q[x]/(f)$ which contain ``unit discriminants~$\Delta\not=1$'', i.e.~dicriminants
    such that $\Delta/\id f_\Delta^2$ is trivial. The columns list the signature,
    discriminant, defining polynomial~$f$, elementary divisors of the
    class group and the narrow class group, and the unit discriminants (modulo squares) of $K$.
    Discriminants which are not totally positive are marked by a ${}^*$.}
  \footnotesize
  \newcolumntype{R}{>{$}r<{$}}
  \renewcommand{\arraystretch}{1.2} 

\begin{tabular}{@{} RRRRRr @{}} \toprule\text{\bf sign.}&\boldsymbol{D_K}&\boldsymbol{f}&\boldsymbol{\Cl(K)}&\boldsymbol{\Cl^+(K)}&\text{\bf discriminants}\\ \midrule  0,1&-15&x^{2} - x + 4&\left[2\right]&\left[2\right]&\parbox[t]{6cm}{\raggedright ${1}$, ${-a + 1}$}\\ &-84&x^{2} + 21&\left[2, 2\right]&\left[2, 2\right]&\parbox[t]{6cm}{\raggedright ${1}$, ${-1}$, ${3}$, ${-3}$}\\ &-420&x^{2} + 105&\left[2, 2, 2\right]&\left[2, 2, 2\right]&\parbox[t]{6cm}{\raggedright ${1}$, ${-1}$, ${3}$, ${-3}$, ${5}$, ${-5}$, ${7}$, ${-7}$}\\ \midrule 2,0&40&x^{2} - 10&\left[2\right]&\left[2\right]&\parbox[t]{6cm}{\raggedright ${1}$, ${2}$}\\ &60&x^{2} - 15&\left[2\right]&\left[2, 2\right]&\parbox[t]{6cm}{\raggedright ${1}$, ${-1}^*$, ${2 a + 8}$, ${-2 a - 8}^*$}\\ &780&x^{2} - 195&\left[2, 2\right]&\left[2, 2, 2\right]&\parbox[t]{6cm}{\raggedright ${1}$, ${-1}^*$, ${2 a + 28}$, ${-2 a - 28}^*$, ${3}$, ${-3}^*$, ${5}$, ${-5}^*$}\\ \midrule 1,1&-283&x^{3} + 4x - 1&\left[2\right]&\left[2\right]&\parbox[t]{6cm}{\raggedright ${1}$, ${a}$}\\ &-6571&x^{3} - x^{2} - 9x - 16&\left[2, 2\right]&\left[2, 2\right]&\parbox[t]{6cm}{\raggedright ${1}$, ${-3 a^{2} - 8 a + 84}$, ${a + 1}$, ${4 a^{2} - 8 a - 35}$}\\ &-300551&x^{3} - 49x - 169&\left[2, 2, 2\right]&\left[2, 2, 2\right]&\parbox[t]{6cm}{\raggedright ${1}$, ${a^{2} - 4 a - 36}$, ${a + 3}$, ${-a^{2} + a + 61}$, ${a + 7}$, ${3 a^{2} - 15 a - 83}$, ${a + 11}$, ${7 a^{2} - 31 a - 227}$}\\ \midrule 3,0&1957&x^{3} - x^{2} - 9x + 10&\left[2\right]&\left[4\right]&\parbox[t]{6cm}{\raggedright ${1}$, ${-a + 3}$}\\ &7537&x^{3} - x^{2} - 24x - 35&\left[2\right]&\left[2, 2\right]&\parbox[t]{6cm}{\raggedright ${1}$, ${-a - 3}^*$, ${2 a^{2} + 9 a + 10}$, ${-17 a^{2} - 85 a - 100}^*$}\\ &210649&x^{3} - x^{2} - 140x - 587&\left[2, 2\right]&\left[2, 2, 2\right]&\parbox[t]{6cm}{\raggedright ${1}$, ${-a - 7}^*$, ${2 a^{2} + 25 a + 78}$, ${-41 a^{2} - 533 a - 1720}^*$, ${a^{2} - 6 a - 71}$, ${-2 a^{2} - 27 a - 90}^*$, ${-5 a - 31}^*$, ${5 a^{2} + 66 a + 217}$}\\ \midrule 0,2&1521&x^{4} - x^{3} + 4x^{2} + 3x + 9&\left[2\right]&\left[2\right]&\parbox[t]{6cm}{\raggedright ${1}$, ${\frac{1}{4} a^{3} + \frac{3}{4}}$}\\ &18000&x^{4} + 15x^{2} + 45&\left[2, 2\right]&\left[2, 2\right]&\parbox[t]{6cm}{\raggedright ${1}$, ${-1}$, ${-\frac{1}{3} a^{2}}$, ${\frac{1}{3} a^{2}}$}\\ &112896&x^{4} + 20x^{2} + 121&\left[2, 2, 2\right]&\left[2, 2, 2\right]&\parbox[t]{6cm}{\raggedright ${1}$, ${-1}$, ${\frac{2}{11} a^{3} + \frac{18}{11} a + 1}$, ${-\frac{2}{11} a^{3} - \frac{18}{11} a - 1}$, ${-\frac{2}{11} a^{3} - \frac{18}{11} a + 1}$, ${\frac{2}{11} a^{3} + \frac{18}{11} a - 1}$, ${3}$, ${-3}$}\\ \midrule 2,1&-6848&x^{4} - 2x^{3} + 5x^{2} - 2x - 1&\left[2\right]&\left[2\right]&\parbox[t]{6cm}{\raggedright ${1}$, ${-a^{2} + 1}$}\\ &-12375&x^{4} - x^{3} - 4x^{2} - 11x - 29&\left[2\right]&\left[2, 2\right]&\parbox[t]{6cm}{\raggedright ${1}$, ${-\frac{3}{13} a^{3} - \frac{8}{13} a^{2} - a - \frac{19}{13}}^*$, ${-\frac{1}{26} a^{3} + \frac{3}{13} a^{2} + a + \frac{37}{26}}$, ${-\frac{37}{26} a^{3} - \frac{45}{13} a^{2} - 7 a - \frac{295}{26}}^*$}\\ &-198000&x^{4} + 15x^{2} - 495&\left[2, 2\right]&\left[2, 2, 2\right]&\parbox[t]{6cm}{\raggedright ${1}$, ${-1}^*$, ${\frac{2}{7} a^{3} + \frac{8}{7} a^{2} + \frac{64}{7} a + \frac{256}{7}}$, ${-\frac{2}{7} a^{3} - \frac{8}{7} a^{2} - \frac{64}{7} a - \frac{256}{7}}^*$, ${-\frac{1}{21} a^{2} + \frac{15}{7}}$, ${\frac{1}{21} a^{2} - \frac{15}{7}}^*$, ${3}$, ${-3}^*$}\\ \midrule 4,0&21025&x^{4} - 17x^{2} + 36&\left[2\right]&\left[2\right]&\parbox[t]{6cm}{\raggedright ${1}$, ${-\frac{1}{12} a^{3} + \frac{1}{2} a^{2} + \frac{5}{12} a - \frac{1}{2}}$}\\ &32625&x^{4} - x^{3} - 19x^{2} + 4x + 76&\left[2\right]&\left[2, 2\right]&\parbox[t]{6cm}{\raggedright ${1}$, ${-\frac{1}{4} a^{3} - \frac{1}{4} a^{2} + \frac{17}{4} a + \frac{17}{2}}$, ${2 a^{2} - 3 a - 26}^*$, ${6 a^{3} + 8 a^{2} - 97 a - 202}^*$}\\ &176400&x^{4} - 19x^{2} + 64&\left[2\right]&\left[2, 2, 4\right]&\parbox[t]{6cm}{\raggedright ${1}$, ${-1}^*$, ${\frac{1}{2} a^{3} - 2 a^{2} - \frac{3}{2} a + 7}^*$, ${-\frac{1}{2} a^{3} + 2 a^{2} + \frac{3}{2} a - 7}$, ${-\frac{1}{4} a^{3} - a^{2} + \frac{3}{4} a + 4}^*$, ${\frac{1}{4} a^{3} + a^{2} - \frac{3}{4} a - 4}^*$, ${\frac{1}{4} a^{3} - a^{2} - \frac{3}{4} a + 4}^*$, ${-\frac{1}{4} a^{3} + a^{2} + \frac{3}{4} a - 4}^*$}\\ \midrule 1,2&41381&x^{5} - x^{4} - 2x^{2} + 4x - 1&\left[2\right]&\left[2\right]&\parbox[t]{6cm}{\raggedright ${1}$, ${a}$}\\ \midrule 3,1&-243219&x^{5} - 2x^{4} + 2x^{3} - 12x^{2} + 21x - 9&\left[2\right]&\left[2\right]&\parbox[t]{6cm}{\raggedright ${1}$, ${a^{4} + 2 a^{2} - 9 a + 5}$}\\ &-802663&x^{5} - x^{4} - 5x^{3} - 6x^{2} - 3x + 1&\left[2\right]&\left[2, 2\right]&\parbox[t]{6cm}{\raggedright ${1}$, ${a^{3} - 2 a^{2} - 4 a}^*$, ${a + 2}$, ${a^{4} - 8 a^{2} - 8 a}^*$}\\ \midrule 5,0&4010276&x^{5} - 11x^{3} - 9x^{2} + 14x + 9&\left[2\right]&\left[4\right]&\parbox[t]{6cm}{\raggedright ${1}$, ${a^{3} + 2 a^{2} - a - 1}$}\\ &5229109&x^{5} - x^{4} - 12x^{3} + 26x^{2} - 12x - 1&\left[2\right]&\left[2, 2\right]&\parbox[t]{6cm}{\raggedright ${1}$, ${a}^*$, ${a^{4} + a^{3} - 9 a^{2} + 6 a + 2}$, ${2 a^{4} + 3 a^{3} - 20 a^{2} + 14 a + 1}^*$}\\ &15216977&x^{5} - 2x^{4} - 9x^{3} + 2x^{2} + 8x - 1&\left[2\right]&\left[2, 2, 2\right]&\parbox[t]{6cm}{\raggedright ${1}$, ${a^{2} + a}$, ${3 a^{4} - 12 a^{3} - 3 a^{2} + 13 a - 2}^*$, ${6 a^{4} - 23 a^{3} - 7 a^{2} + 25 a - 3}^*$, ${-4 a^{4} + 16 a^{3} + 5 a^{2} - 18 a + 1}^*$, ${-8 a^{4} + 32 a^{3} + 9 a^{2} - 35 a + 4}^*$, ${a^{3} - 2 a^{2} - a + 2}^*$, ${a^{3} - 3 a^{2} + 2 a}^*$}\\ \bottomrule \end{tabular}

\label{tab:unit-dicriminants}\end{adjustwidth}
\end{table}

\FloatBarrier

\section{Appendix: Hilbert symbol and higher unit groups}
\label{sec:Hilbert-symbol}

We give a short self-contained proof of a property of the Hilbert
symbol in dyadic number fields\footnote{We thank Chandan Singh Dalawat
  for a helpful discussion on this matter~\cite{mathoverflow}.}
(Theorem~\ref{thm:dyadic-hilbert-symbol} below) which we needed for
the calculation of the conductors of the characters $\leg \Delta*$ in
Section~\ref{sec:proof-groessencharakter-property}.
 
In this section $K$ denotes a finite extension of $\Q_2$ with ring of
integers $\D$ and prime element~$\pi$. We let $e$ be the ramification
index of~$K$, i.e.~the largest integer such that $\pi^e\mid 2$. We use
$\hilb K\_\_$ for the quadratic Hilbert symbol of~$K$. Recall that
this is the map $\units K\times \units K\rightarrow \{\pm 1\}$ such
that $\hilb Kab=1$ if and only if $ax^2+by^2=1$ has solutions $x$, $y$
in $K$.  The Hilbert symbol is bilinear (see e.g.~\cite[Prop.~57:10
and p.~166]{OMeara}), it obviously factors through a bilinear form on
$\units K/{\units K}^2$, and this form is non-degenerate,
i.e.~$\hilb Kab=1$ for all $b$ is only possible if $a$ is a square in
$K$ (see e.g.~\cite[63:13]{OMeara} for a short proof).

We set $U_0:=\units \D$, and, for $n\ge 1$, let $U_n=1+\pi^n\D$ be the
$n$th higher unit group of $K$. Clearly $U_{n}\supseteq U_{n+1}$ and
$U_k^2\subseteq U_{2k}^2$.  The Local Square Theorem states that
$U_{2e+1}=U_{e+1}^2$. (Recall a simple proof:
$1+4\pi X=\big(\sum_{n\ge 0}\binom {1/2}n (4\pi X)^n\big)^2$ in the
ring of formal power series $\D[\![X]\!]$, and the series converges
for each $x$ in~$\D$ with respect to the valuation~$v$ of~$K$ (and
then towards an element of $U_{e+1}$) since
$v(\binom {1/2}n (4\pi x)^n) \ge n(e+1)-v(n!)$ and, by Legendre's
formula, $v(n!) = (n-s_n)e$, where $s_n$ is the sums of~$1$s is the
binary expansion of $n$, so that $v(n!)\le (n-1)e$.)

\begin{Lemma}
  \label{lem:units-mod-squares}
  For $0\le k\le e-1$, one has
  \begin{equation*}
    U_{2k} = U_{2k+1} U_k^2
    .
  \end{equation*}
\end{Lemma}

\begin{proof}
  Let $a$ in $U_{2k}$. The congruence
  $a \equiv \left(1+\pi^{k}t\right)^2 \bmod \pi^{2k+1}$ is equivalent
  to $t^2 \equiv \left({a-1}\right)/{\pi^{2k}}\bmod \pi$ (since the
  assumption $k<e$ implies $\pi^{2k+1}\mid 2\pi^{k}$), and this has a
  solution~$t$ (as the map $t\mapsto t^2$ defines an automorphism of
  $\D/\pi$).  With such a solution $t$, we have
  $a / \left(1+\pi^{k}t\right)^2 \equiv 1 \bmod \pi^{2k+1}$,
  i.e.~$a\in U_{2k+1}\left(1+\pi^{k}t\right)^2$. For proving the
  inverse inclusion it suffices to note that $U_k^2\subseteq U_{2k}$.
  Indeed, for any $(1+\pi^{k}t)\in U_k$, we have
  $(1+\pi^{k}t)^2\equiv 1\bmod \pi^{2k}$, since $\pi^k\mid 2$.
\end{proof}

\begin{Proposition}
  \label{prop:Hilbert-symbol-on-unit-groups}
  For any pair of even integers $i,j\ge 0$ with $i+j=2e$, one has
  $\hilb K {U_i}{U_j} =1$.
\end{Proposition}

\begin{proof}
  We can assume that $0\le i\le e$.  Let $a$ in $U_{i}$ and $b$ in
  $U_{j}$. We have to show that $ax^2+by^2=1$ has solutions $x$ and
  $y$ in $K$.  We can assume by Lemma~\ref{lem:units-mod-squares} that
  $a$ is in~$U_{i+1}$. But then $a+b-ab = 1-(a-1)(b-1)$ is in
  $U_{2e+1}$, hence, by the Local Square Theorem $a+b-ab=w^2$ for a
  unit $w$. But then again $a(w+b)^2+b(w-a)^2=(a+b)^2$, which implies
  the claim.
\end{proof}

\begin{Lemma}
  \label{lem:U2e}
  One has $U_{2e}/U_{e}^2\cong \F_2$.
\end{Lemma}
\begin{proof}
  The application $1+4x\mapsto \tr_{\D/\pi/\F_2} \overline x$ (where
  $\overline x$ is the residue class of $x$ mod~$\pi$) defines an
  epimorphism of $U_{2e}$ onto $\F_2$. We claim that its kernel equals
  $U_e^2$. It contains $U_e^2$ since Since $(1+ey)^2=1+4(y+y^2)$ and
  $\tr_{\D/\pi/\F_2}(\overline y+\overline y^2)=0$. Vice versa, if
  $1+4x$ has $\tr_{\D/\pi/\F_2} \overline x = 0$ then
  $1+4x\equiv (1+2y)^2\bmod 4\pi$, i.e.~$x\equiv y+y^2\bmod \pi$, or
  equivalently $(1+4x)/(1+2y)^2\in U_{2e+1}=U_{e+1}^2$, has a
  solution~$y$. Namely,
  $\overline y\mapsto \overline y + \overline y^2$ defines an
  $\F_2$-linear endomorphism of~ $\D/\pi$ with kernel~$\F_2$ and image
  equal to the the kernel of~$\tr_{(\D/\pi)/\F_2}$.
\end{proof}

We use $\overline U_n$ for the image of $U_n$ under the canonical map
$\units K\rightarrow \units K/{\units K}^2$. Note
that~$\overline U_n=U_n{\units K}^2/{\units K}^2$.  for
$0\le k\le e-1$, Lemma~\ref{lem:units-mod-squares} implies
$U_{2k}{\units K}^2 = U_{2k+1}{\units K}^2$, and we conclude
\begin{equation*}
  \overline U_{2k}=\overline U_{2k+1}
  .
\end{equation*}
We set
\begin{equation}
  \label{eq:filtrations-sequence}
  V_{-1}:=\units K/{\units K}^2,\quad
  V_k := U_{2k}{\units K}^2/{\units K}^2
  \quad (0\le k\le e),
  \quad
  V_{e+1}:=1
  .
\end{equation}

\begin{Lemma}
  One has
  \begin{equation*}
    1=V_{e+1}\subseteq_2
    V_{e} \subseteq_{2^f} V_{e-1}
    \dots \subseteq_{2^f} V_0
    \subseteq_2 V_{-1} = \units K/{\units K}^2
    ,
  \end{equation*}
  where ``$\subseteq_n$'' means ``is subgroup of index $n$'', and
  where $f$ is the degree of~$\pi$, i.e.~$2^f=\card{\D/\pi}$.  One
  has, in particular, for $0\le k\le e$,
  \begin{equation}
    \label{eq: higher-unis-cardinalities}
    \card V_k = 2\cdot 2^{f(e-k)}
    .
  \end{equation}
\end{Lemma}

\begin{proof}
  The first equality in the filtration chain, i.e.~that the elements
  of $U_{2e+1}$ are squares, is the Local Square Theorem.

  For the first ``$\subseteq_2$'' we note that
  $U_{2e}\cap {\units K}^2= U_e^2$, whence
  \begin{equation*}
    U_{2e}/U_e^2 \cong U_{2e}{\units K}^2/{\units K}^2
    ,
  \end{equation*}
  and apply Lemma~\ref{lem:U2e}.

  For the last ``$\subseteq_2$'' note that
  $\units K=\langle\pi\rangle\times U_0$, from which we deduce
  $\units K/{\units K}^2 = \langle\overline \pi\rangle\times \overline
  U_0$, where $\overline {(\ )}$ is the canonical projection.
 
  For the ``$\subseteq_{2^f}$'' we calculate
  \begin{equation*}
    V_{k-1}/V_k
    =
    \overline U_{2k-1}/\overline U_{2k}
    \cong U_{2k-1}{\units K}^2/U_{2k}{\units K}^2
    \cong U_{2k-1}/U_{2k} \cong \D/\pi\D
  \end{equation*}
  The last two isomorphisms are from right to left:
  $a+\pi\D \mapsto (1+a\pi^{2k-1})U_{2k}$ and
  $uU_{2k}\mapsto uU_{2k}{\units K}^2$.  Note that the second
  application defines indeed an isomorphism: It is obviously
  surjective. For proving that it is injective suppose, $uU_{2k}$, for
  a given $u$ in $U_{2k-1}$, is mapped to $U_{2k}{\units K}^2$,
  i.e.~$u=va^2$ for some $v$ in $U_{2k}$ and $a$ in~$\units K$.  But
  then $a$ must be unit, and since $a^2\equiv 1\bmod \pi^{2k-1}$ and
  $k\le e$, we conclude that, in fact, $a\equiv 1\bmod \pi^k$, which
  in turn implies that $a^2$ is in $U_{2k}$ (again since $k\le e$);
  hence $u=va^2$ is in $U_{2k}$.
\end{proof}

\begin{Theorem}
  \label{thm:dyadic-hilbert-symbol}
  For the subspaces of the filtration~\eqref{eq:filtrations-sequence}
  of $\units K/{\units K}^2$, one has
  \begin{equation*}
    \dual V_k = V_{e-k}
    ,
  \end{equation*}
  for every $-1\le k\le e+1$ (Here $\dual V_k$ denotes the subgroup of
  all $a{\units K}^2$ in $V_{-1}=\units K/{\units K}^2$ such that
  $\hilb K ab=1$ for all $b{\units K}^2$ in $V_k$.)
\end{Theorem}
    
\begin{proof}
  The Hilbert symbol defines a non-degenerate bilinear form on the
  $\F_2$-vector space $V_{-1}$.  From~\eqref{eq:
    higher-unis-cardinalities} we know
  \begin{equation*}
    \dim_{\F_2} V_k + \dim_{\F_2} V_{e-k}=\dim_{\F_2} V_{-1}
    ,
  \end{equation*}
  and hence $\dim_{\F_2} \dual V_k = \dim_{\F_2} V_{e-k}$.  It
  suffices therefore to prove $\dual V_k \supseteq V_{e-k}$. For
  $k={-1}$ or $k=e+1$ this is trivial. For $0\le k\le e$ the inclusion
  is equivalent to~$\hilb K{U_{2k}}{U_{2e-2k}}=1$. But this is
  Proposition~\ref{prop:Hilbert-symbol-on-unit-groups}.
\end{proof}

\bibliography{qfields.bib}
\addcontentsline{toc}{chapter}{Bibliography} \bibliographystyle{alpha}

\end{document}